\documentclass[12pt]{amsart}
\usepackage{amsmath,amsthm,amsfonts, amssymb,graphicx,amscd,float,hyperref,enumitem,setspace,picins}
\usepackage{lipsum}
\usepackage[usenames, dvipsnames]{color}
 \usepackage{geometry}
  \usepackage{accents}
\usepackage{epsfig}
\usepackage{marginnote}
\usepackage{comment}
\usepackage{xypic}

  \usepackage{xspace,verbatim}
\usepackage{epstopdf,wrapfig}

\newtheorem{thm}{Theorem}[section]
\newtheorem{theorem}[thm]{Theorem}
\newtheorem{lem}[thm]{Lemma}
\newtheorem{lemma}[thm]{Lemma}
\newtheorem{cor}[thm]{Corollary}

\newtheorem{qst}[thm]{Question}

\newtheorem{prop}[thm]{Proposition}

\theoremstyle{definition}
\newtheorem{df}[thm]{Definition}
\newtheorem{definition}[thm]{Definition}

\newtheorem{rk}[thm]{Remark}
\newtheorem{obs}[thm]{Observation}
\newtheorem{remark}[thm]{Remark}
\newtheorem*{mainthmA}{Theorem A}
\newtheorem*{mainthmB}{Theorem B}
\newtheorem*{mainthmC}{Theorem C}
\newtheorem*{mainthmD}{Theorem D}

  \newcommand{\calF}{\mathcal{F}}

  \newcommand{\calO}{\mathcal{O}}

  \newcommand{\calR}{\mathcal{R}}
  \newcommand{\calS}{\mathcal{S}}
  \newcommand{\calT}{\mathcal{T}}

  \newcommand{\calX}{\mathcal{X}}

  \newcommand{\NN}{\mathbb{N}}

  \newcommand{\RR}{\mathbb{R}}
  \renewcommand{\SS}{\mathbb{S}}

\newcommand{\out}{\textup{Out}(F_r)}

\newcommand{\oOmega}{\overline{\Omega}}
\newcommand{\oT}{\overline{T}}
\newcommand{\oK}{\overline{K}}
\newcommand{\ol}[1]{\overline{#1}}

\newcommand{\pc}[1]{C^{#1}}

\newcommand{\sref}[1]{\S\ref{#1}}

\makeatletter
\newcommand*{\inlineequation}[2][]{%
  \begingroup
    \refstepcounter{equation}%
    \ifx\\#1\\%
    \else
      \label{#1}%
    \fi
    \relpenalty=10000 %
    \binoppenalty=10000 %
    \ensuremath{%
      #2%
    }%
    ~\@eqnnum
  \endgroup
}
\makeatother

   \newcommand{\mF}{{\mathcal F}}
  \newcommand{\mS}{{\mathcal S}}
 \newcommand{\mP}{{\mathcal P}}
 \newcommand{\mR}{{\mathcal R}}
   
  \newcommand{\A}{\mathcal{A}}
 \newcommand{\B}{\mathcal{B}}

\newcommand{\dss}{Y_r} 

\newcommand{\mauto}{g}
\newcommand{\rsimp}{\sig_0}

\renewcommand{\span}{\text{span}}

\renewcommand{\b}{f}
\newcommand{\ov}[1]{\overrightarrow{#1}}

  \newcommand{\param}{{\mathchoice{\mkern1mu\mbox{\raise2.2pt\hbox{$
  \centerdot$}}
  \mkern1mu}{\mkern1mu\mbox{\raise2.2pt\hbox{$\centerdot$}}\mkern1mu}{
  \mkern1.5mu\centerdot\mkern1.5mu}{\mkern1.5mu\centerdot\mkern1.5mu}}}

\newcommand{\os}{\mathcal{X}_r}
\newcommand{\ros}{\mathcal{RX}_r}
\newcommand{\uos}{\hat{\mathcal{X}}_r}
\newcommand{\uros}{\mathcal{URX}_r}
\newcommand{\ms}{\mathcal{M}_r}
\newcommand{\teich}{Teichm\"{u}ller }
\newcommand{\MCG}{\text{MCG} }

\newcommand{\eps}{\varepsilon}
\newcommand{\lip}[1]{\text{Lip}(#1)}
\newcommand{\al}{\alpha}
\newcommand{\from}{\colon}
\newcommand{\sig}{\sigma}

\newcommand{\st}[3]{\text{st}_{#1}(#2,#3) }

\newcommand{\im}{\text{Im}}

\newcommand{\far}{\rightharpoonup} 

\begin{document}
\title{A dense geodesic ray in the $\out$-quotient of reduced Outer Space.}
\author{Yael Algom-Kfir, Catherine Pfaff}

\address{\tt Department of Mathematics, University of Haifa \newline
\indent  {\url{http://www.math.haifa.ac.il/algomkfir/}}, } \email{\tt yalgom@univ.haifa.ac.il}
  
\address{\tt Department of Mathematics, University of California at Santa Barbara \newline
  \indent  {\url{http://math.ucsb.edu/~cpfaff/}}, } \email{\tt cpfaff@math.ucsb.edu}

\begin{abstract}
In \cite{m81} Masur proved the existence of a dense geodesic in the moduli space for a Teichm\"{u}ller space. We prove an analogue theorem for reduced Outer Space endowed with the Lipschitz metric. We also prove two results possibly of independent interest: we show Brun's unordered algorithm weakly converges and from this prove that the set of Perron-Frobenius eigenvectors of positive integer $m \times m$ matrices is dense in the positive cone $\RR^m_+$ (these matrices will in fact be the transition matrices of positive automorphisms). We give a proof in the appendix that not every point in the boundary of Outer Space is the limit of a flow line.
\end{abstract}

\thanks{\tiny{This research was supported by THE ISRAEL SCIENCE FOUNDATION (grant no. 1941/14). The second author was supported first by the ARCHIMEDE Labex (ANR-11-LABX- 0033) and the A*MIDEX project (ANR-11-IDEX-0001-02) funded by the ``Investissements d'Avenir,'' managed by the ANR. She was secondly supported by the CRC701 grant of the DFG, supporting the projects B1 and C13 in Bielefeld.}} 

\maketitle

\section{Introduction}

\subsection{Geodesics in Outer Space}

One of the richest and most expansive methods for studying surfaces has been through the ergodic geodesic flow on \teich space. 
As an example, it was used by Eskin and Mirzakhani  \cite{eskinMirCounting} to count pseudo-Anosov conjugacy classes 
of a bounded length. For this reason, the papers of Masur \cite{m82} and Veech \cite{v82} independently proving the ergodicity of the \teich flow were seminal in the field. 
The existence of an $\out$-invariant ergodic geodesic flow on 
 Outer Space may similarly expand the tools for studying $\out$. 
Before giving the proof of the ergodicity theorem in \teich space, Masur performed the following a ``litmus test'' for its plausibility.

\begin{thm}[\cite{m81}]\label{thm_Mas1}
Given a closed surface $S_g$ of genus $g$, there exists a \teich geodesic in the \teich space $T_g$ whose projection $p : T_g \to T_g/\MCG(S_g)$ to moduli space is dense in both directions. 
\end{thm} 

Our main theorem is an $\out$ analogue of the above theorem. Some of the terms in the theorem are defined directly below its statement.

\begin{mainthmA}\label{main1}
For each $r \geq 2$, there exists a geodesic fold ray in the reduced Outer Space $\ros$ whose projection to $\ros/\out$ is dense. 
\end{mainthmA}

\begin{rk}\label{r:mainthmrk}
\begin{enumerate}
\item The reduced Outer Space $\ros$ is a subcomplex of the Outer Space $\os$, which consists of those graphs without separating edges (see Definition \ref{d:ros}). It is an equivariant deformation retract of $\os$. 
\item The metric on $\ros$ with respect to which the ray in Theorem A is a geodesic is the Lipschitz metric (see Definition \ref{d:LipMet}). It is an asymmetric metric (analogous to the Thurston metric  on \teich space \cite{ThurstonMetric}) that has proven to be very useful in the $\out$ context, e.g. \cite{BestBers}. 
\item Because of the asymmetry of the metric, our geodesics will always be ``directed geodesics," i.e. maps 
$\Gamma \from [0, \infty) \to \os$
such that $d(\Gamma(t), \Gamma(t')) = t'-t$ for $t'>t\geq 0$, but not necessarily for $t>t'$. 
\item A fold line is a special kind of geodesic in $\os$ (explicitly described in Definition \ref{FoldPath}) that is analogous to a \teich geodesic. 
\item Comparing between Theorem \ref{thm_Mas1} and Theorem A, one may notice that our theorem declares the existence of a ray in contrast with Masur's theorem which declares the existence of a geodesic. We could easily extend our ray to a bi-infinite geodesic. However, the density of the image of the ray will follow from techniques that we cannot extend to the backwards direction. 
\end{enumerate}
\end{rk}

Returning to Remark \ref{r:mainthmrk}(1), we note that for proving algebraic properties of $\out$, one may usually replace Outer Space with $\ros$. However, on the geometric side, it is not known whether or not $\ros$ is convex in any coarse sense. 

We hence pose two questions:

\begin{qst}
For each $r \geq 2$, does there exist a geodesic fold line in $\os$ that is dense in both directions in $\os/\out$?
\end{qst}

\begin{qst}
For each $r \geq 2$, is the reduced Outer Space $\ros$ coarsely convex? 
For example, given points $x,y \in \ros$ does there always exist a geodesic from $x$ to $y$ which is contained in reduced Outer Space?
\end{qst}

\subsection{The unit tangent bundle}

Masur obtained Theorem \ref{thm_Mas1} as a corollary of his study of the geodesic flow on \teich Space. 
The unit tangent bundle of \teich space is isomorphic to its unit cotangent bundle $Q_0$, which may be described explicitly as the space of unit area quadratic differentials on a closed surface $S_g$ of genus $g$. The geodesic flow on \teich space is a $\MCG(S_g)$ invariant action of $\RR$ on $Q_0$. For $t \in \RR$ we denote the flow by $T_t:Q_0 \to Q_0$. 
Given a quadratic differential $q \in Q_0$, the set of points $\{ q_t \}_{t \in \RR}= \{ T_t q \}_{t \in \RR}$ defines a geodesic in the \teich space $T_g$. 

\begin{thm}[\cite{m81}]\label{thm_Mas2}
For a closed surface $S_g$ of  genus $g>1$, 
there exists a quadratic differential $q \in Q_0$ so that the projection of $\{q_t\}$ for either $t>0$ or $t<0$ is dense in $Q_0/\MCG(S_g)$.
\end{thm}

The analogue of Theorem \ref{thm_Mas2} is not obvious, as it is unclear what  
the unit tangent bundle should be. In Outer Space the analogues of the following facts are false:
\begin{enumerate}
\item If $\Gamma_1, \Gamma_2$ are two bi-infinite \teich geodesics satisfying that there exist non-degenerate intervals $I,I' \subset \RR$ with $\Gamma_1(I) \subset \Gamma_2(I')$ then, up to reparametrization, $\Gamma_1 = \Gamma_2$. 

\item There exists a compactification $\overline{T_g}$ of $T_g$ 
such that for any two points $x,y \in \overline{T_g}$ there exists a geodesic from $x$ to $y$. 
\end{enumerate}

The failure of (1) is an impediment to a ``local" description of the tangent bundle. This failure is
similar to its failure in a simplicial tree. 
Two geodesics in Outer Space can meet for a period of time and then diverge from each other (or even alternately meet and diverge). 
On the other hand, the failure of (2) is an impediment to a ``global" description in terms of the boundary of Outer Space. One may want to use equivalence classes of rays emanating from a common base-point.
However, we show in \S \ref{AppendixBoundary} that there are points on the boundary i.e. $\overline{\os} - \os$ that are not ends of geodesic fold rays.  
(While this result may be known to the experts, to our knowledge this is the first time it appears in print).
Thus one would first have to identify the visual boundary as a subset of $\overline{\os} - \os$ and this has yet to be done. 

For the purposes of this paper, we propose the following analogue of the unit tangent bundle. 
Given a point $x \in \ros$ there are finitely many germs $[\al]$ of fold lines $\al$ in $\ros$ initiating at $x$. Define 
\[ \uros = \{ (x,[\al]) \mid x \in \ros, \al \text{ is a fold line with } \al(0)=x \text{ and } \im(\al) \subset \ros \}. \]
Given a geodesic ray $\gamma \from [0, \infty) \to \ros$, for $t \in \RR$ we denote by $\gamma_t$ the path $\gamma_t(s) = \gamma(t+s)$. We may associate to $\gamma$ a path in the unit tangent bundle $\tilde \gamma(t) = (\gamma(t), [\gamma_t])$. We prove:

\begin{mainthmB}
For each $r \geq 2$, there exists a Lipschitz geodesic fold ray $\tilde \gamma \from [0, \infty) \to \ros$ so  that the projection of $\tilde \gamma$ to $\uros/\out$ is dense. 
\end{mainthmB}

\subsection{Geodesics in other subcomplexes}

\smallskip 
For each $r \geq 2$, we define the \emph{theta subcomplex} $\mathcal{T}_r$ to be the subspace of $\ros$ consisting of all points in $\os$ whose underlying graph is either a rose or a theta graph, see Figure \ref{thetaGraph1}. 
This subcomplex carries the significance of being the minimal subcomplex containing the image of the Cayley graph under the natural map.
Both as a warm-up, and for its intrinsic significance, we initially prove Theorem A in $\mathcal{T}_r$.

\begin{figure}[h]
\centering
\includegraphics
{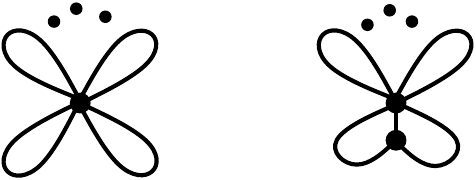}
\caption{ The\label{thetaGraph1} underlying graphs of the simplices of $\mathcal{T}_r$. The graph on the left is called a \emph{rose} and that on the right is called a \emph{theta graph}.}
\end{figure}

\begin{mainthmC}\label{theta_comp_thm}
For each $r \geq 2$, there exists a fold line in $\mathcal{T}_r$ that projects to a Lipschitz geodesic fold ray in $\mathcal{T}_r/\out$. 
\end{mainthmC}

\subsection{Outline}

We begin by outlining the proof of Theorem C. After proving Theorem C, we develop the topological machinery necessary to extend the more combinatorial arguments of the proof of Theorem C to the proof of Theorem B (and Theorem A as a corollary).

Recall that points in Outer Space are marked metric graphs 
equivalent up to homotopy, see Definition \ref{d:markedmetricgraph}. As described in Definition \ref{d:OutAction}, $\out$ acts on the right by changing the marking. To a point $x \in \mathcal{T}_r/\out$,  one may associate a positive vector, the ``length vector" recording the graph's edge lengths. The folding operation may be translated to a matrix recording the change in edge lengths from its initial point to its terminal point.
In this dictionary, a fold ray in $\mathcal{T}_r/\out$ should correspond to an initial vector and a sequence of fold matrices.
However, not every such sequence comes from a fold ray:  
a particular fold may or may not be allowed for a specific vector depending on whether its image is again a positive vector.
Our challenge is to construct a sequence of fold matrices
$\{T_k\}_{k=1}^\infty$ 
satisfying that, for some positive vector $w_0$, if we write $w_i = T_{i} \cdots T_1(w_0)$ for each $i \in \NN$ then 
\begin{itemize}
\item[I.] \label{allowable333} the fold $T_{i+1}$ is allowed in $w_i = T_{i} \cdots T_1(w_0)$ for each integer $1 \leq i < \infty$, 
\item[II.] \label{dense333} the set of vectors $w_i$ is projectively dense in a simplex and, 
\item[III.] \label{geodesic333} the corresponding fold ray is a geodesic ray in the Lipschitz metric.
\end{itemize}

In order to prove item (II) of the list, we prove the following fact: 

\begin{mainthmD}\label{thm_PF_dense}
For each $r \geq 2$, let $S_{l_1}^r$ be the set of unit vectors according to the $l_1$ metric in $\RR^r_+$. The set of Perron-Frobenius eigenvectors of matrices arising as the transition matrices of a positive automorphisms in $\text{Aut}(F_r)$ is dense in the $S_{l_1}^r$.
\end{mainthmD}

The proof of Theorem D, in \sref{s:BrunsAlgorithm}, uses Brun's algorithm \cite{b57}. We also prove in \sref{s:BrunsAlgorithm} that Brun's (unordered) algorithm converges in angle, a result to our knowledge previously absent from the literature in dimensions higher than four. (Brun proved it in \cite{b57} for dimensions three and four.)  

To construct the sequence $\{T_k\}$ of fold matrices we enumerate the powers of ``Brun matrices" (see \sref{s:BrunsAlgorithm}):
$P_1, P_2, P_3, \dots$.
To each $P_i$, we can attach the following data:
\begin{itemize}
\item a positive Perron-Frobenius eigenvector $v_i$, 
\item a positive automorphism $g_i \in Aut(F_r)$, also denoted $g_{v_i}$, so that the transition matrix of $g_i$ is $P_i$, 
\item and a decomposition of $P_i$ into fold matrices, arising from the decomposition of $g_i$ into Neilsen generators which correspond to moves in Brun's algorithm.
\end{itemize}

We remark that this method is reminiscent of Masur's paper, where he proved the existence of a dense geodesic in $R_g = Q_0/\MCG(S_g)$ using the fact that closed loops in $R_g$ are dense. 
The resemblance stems from the facts that the decomposition in the third item defines a loop in $\ros/\out$ based at $v_i$ and that the density,  established by Theorem~D, of the set of Perron-Frobenious eigenvectors $\{v_i\}$'s. 

We concatenate the fold sequences associated to the matrices $P_i$ together to form the sequence $\{T_k\}_{k=1}^\infty$. We address item  (I) on the list, i.e. the allowability of the sequence, in Lemma \ref{BasePoint}. We now have a fold ray through the points $\{w_j\}_{j=1}^\infty$. 

For density of the geodesic ray we use the automorphisms $g_i$ related to the $P_i$. By ensuring that arbitrarily high powers of these automorphisms (hence matrices) occur in the sequence, we ensure that the ray passes through points with length vectors arbitrarily close to the dense set of eigenvectors.

Finally, property (III) on the list, that the fold line is a Lipschitz geodesic, follows from the fact that every $g_i$ is a positive automorphism (see Corollary \ref{PositiveFoldLineGeodesics}). 

To extend our proof of Theorem C to Theorems A and B, we prove that for a generic point $y$ in reduced Outer Space, there exist roses $x,z$ and a ``positive" fold line $[x,z]$ remaining in reduced Outer Space and so that $y \in [x,z]$. 
Here by ``positive" we mean that the change of marking from $x$ to $z$ is a positive automorphism. Moreover, if $G$ is the underlying graph of $y$ and $E,E'$ are two adjacent edges in $G$, then one may choose $[x,z]$ so that it contains the fold of the turn $\{E,E'\}$ immediately after the point $y$. This construction is carried out in \sref{S:dgAll} and elevates the geodesic's density in $\ros/\out$ to density in $\mathcal{U}_r$. Additionally, the geodesic $[x,z]$ varies continuously as a function of $y$, as we prove in \sref{S:ContinuityCloseness}. Thus, one may adjust the previous argument to prove Theorems A and B, which we do in \sref{S:RayConstruction}.

\subsection*{Acknowledgements}
 
We wish to thank Pierre Arnoux for bringing our awareness to Brun's algorithm and its possible use in proving that the Perron-Frobenius eigenvectors are dense, as well as other helpful discussions. On this subject we are also heavily indebted to Jon Chaika who pointed out to us that Brun's algorithm is ergodic and that we could use this to elevate our proof of the Perron-Frobenius eigenvectors being dense in a simplex to additionally prove that Brun's algorithm converges in angle (also somewhat simplifying our proof that the Perron-Frobenius eigenvectors are dense in a simplex). We would further like to thank Jayadev Athreya for posing the question and helpful discussions.  Inspirational ideas given by Pascal Hubert were also particularly valuable. We are indebted to Lee Mosher for pointing out that Keane has a paper on a complication we were facing with regard to Brun's algorithm. And we are indebted to Fritz Schweiger for his generosity in helping us understand arguments of his books. Finally, we are indebted to Val\'erie Berth\'e, Mladen Bestvina, Kai-Uwe Bux, Albert Fisher, Vincent Guirardel, Ilya Kapovich, Amos Nevo, and Stefan Witzel for helpful discussions.

\tableofcontents

\section{Definitions and background}{\label{s:Dfs}}

\subsection{Outer Space and the action of $\out$}

Culler and Vogtmann introduced Outer Space in \cite{cv86}. Points of Outer Space are ``marked metric graphs.''

\begin{df}[Graph, positive edges]
A \emph{graph} will mean a connected 1-dimensional cell complex. $V(G)$ will denote the vertex set and $E(G)$ the set of unoriented edges. 
The degree of a vertex $v \in V(G)$ will be denoted $deg_G(v)$, or $deg(v)$ when $G$ is clear.

For each edge $e \in E(G)$, one may choose an orientation. Once the orientation is fixed that oriented edge $E$ will be called \emph{positive} and the edge with the reverse orientation $\ol{E}$ will be called \emph{negative}. 
 Given an oriented edge $E$,  $i(E)$ will denote its initial vertex and $ter(E)$ its terminal vertex.
A \emph{directed} graph is a graph $G$ with a choice of orientation on each edge $e \in E(G)$, we call this choice an \emph{orientation} on $G$.
\end{df}

Given a free group $F_r$ of rank $r \geq 2$, we choose once and for all a free basis $A=\{X_1, \dots, X_r\}$. Let $R_r = \vee_{i=1}^r \SS^1$ denote the graph with one vertex and $r$ edges, we call this graph an \emph{r-petaled rose}. 
We choose once and for all an orientation on $R_r$ and  
identify each positive edge of $R_r$ with an element of the chosen free basis. Thus, a cyclically reduced word in the basis corresponds to an immersed loop in $R_r$.

\begin{df}[Marked metric $F_r$-graph]{\label{d:markedmetricgraph}} For each integer $r \geq 2$ we define a \emph{marked $F_r$-graph} to be a pair $x =(G,m)$ where:
\begin{itemize}
\item $G$ is a graph such that $deg(v) \geq 3$ for each vertex $v \in V(G)$.
\item $m \colon R_r \to G$ is a homotopy equivalence, called a \emph{marking}.
\end{itemize}

A \emph{marked metric graph} is a triple $(G,m,\ell)$ so that $(G, m)$ is a marked graph and:
\begin{itemize}
\item The map $\ell \colon E(G) \to \RR_+$ is an assignment of lengths to the edges. We require that $\sum_{e \in E(G)} \ell(e) = 1$. The quantity $vol(G) =  \sum_{e \in E(G)} \ell(e)$ is called the \emph{volume} of $G$.
\end{itemize}
 
\end{df}

\begin{rk}[A metric graph as a metric space]\label{metricStructure} The assignment of lengths to the edges does not quite determine a metric on $G$, but a homeomorphism class of metrics.
This choice is inconsequential in this paper.
\end{rk} 

Define an equivalence relation on marked metric $F_r$-graphs by $(G,m, \ell) \sim (G',m', \ell')$ when there exists an isometry $\varphi \colon (G, \ell) \to (G', \ell')$ so that $m'$ is homotopic to $\varphi \circ m$.

\begin{df}[Underlying set of Outer space]
For each $r \geq 2$, as a set, the (\emph{rank-$r$}) \emph{Outer Space} $\os$ is the set of equivalence classes of marked metric $F_r$-graphs.
\end{df}

\begin{rk}
On occasion we may think of graphs with valence-$2$ vertices as living in Outer Space by considering them equivalent to the graphs obtained by unsubdividing at their valence-$2$ vertices.
\end{rk}

\begin{df}\label{simplexDefn}
The \emph{simplex} $\sig$ in $\os$ corresponding to the marked graph $(G,m)$ is
$$\sig_{(G,m)} = \{ (G,m,\ell) \in \os \mid vol(G) =1 \}.$$
By enumerating the edges of $G$, we can identify $\sig_{(G,m)}$ with the open simplex
$$\calS_{|E|} = \left\{ \overrightarrow{v} \in \RR_+^{|E|} ~ \left| ~ \sum_{i=1}^{|E|} v_i = 1 \right. \right\}.$$ 
Here $E = E(G)$. We denote this identification by $n \from \sig_{(G,m)} \to \calS_{|E|}$. 
We call the open simplex corresponding to $(R_r,id)$ the \emph{base simplex} and denote it by $\sig_0$. 
\end{df}

Outer Space has the structure of a simplicial complex built from  open simplices (see \cite{v02}), faces of $\sig_{(G,m)}$ arise by letting the edges of a tree in $G$ have length 0.

\begin{df}[Simplicial metric]
Given a simplex $\sig_{(G,m)}$ in $\mathcal{X}_r$, the \emph{simplicial metric} on $\sig_{(G,m)}$ is defined by  $d_s(\ell,\ell') = \sqrt{\sum_{e \in E(G)}(\ell(e) - \ell'(e))^2}$, 
for $\ell,\ell' \in \sig_{(G,m)}$.
We also denote by $d_s$ the extension of this metric to a path metric on $\mathcal{X}_r$.
\end{df}

\begin{rk}\label{sameTopology}
In \S \ref{SS:Lipschitz} we define the Lipschitz metric on $\os$. The simplicial metric and Lipschitz metric on Outer Space differ in important ways. However, open balls with respect to the Lipschitz metric (in either direction, see  Remark \ref{LipschitzBalls}) contain open balls of the simplicial metric). Therefore, a set dense with respect to the simplicial topology will also be dense with respect to the Lipschitz topology. 
\end{rk}

\begin{df}[Unprojectivized Outer Space]{\label{d:uos}} \cite{cv86}
The \emph{(rank-$r$) unprojectivized Outer Space} $\uos$ is the space of metric marked $F_r$-graphs where $vol(G)$ is not necessarily 1. 

There is a map from 
$\uos$ to $\os$ normalizing the graph volume, i.e. 
\begin{equation}{\label{e:Projectionq}}
\begin{array}{rcl}
q: \RR^m_+ &\to& \mS_{m}\\
q(x_1, \dots, x_m) &=&  \left( \frac{x_1}{\sum_{i=1}^{m}x_i},  \dots, \frac{x_{m}}{\sum_{i=1}^{m}x_i} \right)
\end{array}
\end{equation} 
and taking the point $(G,\mu, \ell)$ to the point $(G, \mu, q(\ell))$. 

We call the full preimage under $q$ of a simplex in $\os$ an \emph{unprojectivized simplex}.
\end{df}

\begin{df}[Topological Outer Space]{\label{d:os}} \cite{cv86}
The topological space consisting of the set of equivalence classes of marked metric $F_r$-graphs, endowed with the simplicial topology, is called the (\emph{rank-$r$}) \emph{Outer Space} and is also denoted by $\os$.
\end{df}

\begin{df}[$\out$ action]{\label{d:OutAction}}
If $\Phi \in \text{Aut}(F_r)$ is an automorphism, let $f_\Phi \colon R_r \to R_r$ be a homotopy equivalence corresponding to $\Phi$ via the identification of $E(R_r)$ with the chosen free basis $A$ of $F_r$. We define a \emph{right action of $\out$} on $\os$. An outer automorphism $[\Phi] \in \out$ acts by $[G,m,\ell]\cdot [\Phi] = [G, m \circ f_\Phi, \ell]$.
\end{df}

\begin{df}[Reduced Outer Space $\ros$ and $\ms$]{\label{d:ros}}
For each integer $r \geq 2$, the \emph{(rank-$r$) reduced Outer Space} $\ros$ is the subcomplex of $\os$ consisting of precisely those simplices $\sig_{(G,m)}$ such that $G$ contains no separating edges. This space is connected and an $\out$-equivariant deformation retract of $\os$.
\end{df}

Let $\ms$ denote the quotient space of $\mathcal{RX}_r$ by the $\out$ action. Hence, $\ms$ contains a quotient of a simplex for each graph (no longer marked). Note that, as a result of graph symmetries, simplices in $\os$ do not necessarily project to simplices in $\ms$. Thus, $\ms$ is no longer a simplicial complex but a union of cells which are quotients of simplices in $\os$.

\subsection{Train track structures}

Much of the following definitions and theory can be found in \cite{bh92} or \cite{b12}, for example. However, it should be noted that some of our definitions, including that of an illegal turn, are somewhat nonstandard.

\begin{df}[Regular maps]
We call a continuous map $f \colon G \to H$ of graphs \emph{regular} if for each edge $e \in E(G)$, we have that $f|_{int(e)}$ is locally injective and that $f$ maps vertices to vertices.
\end{df}

\begin{df}[Paths and loops]\label{defPaths}
Depending on the context an \emph{edge-path} in a graph $G$ will either mean  a continuous map $[0,n] \to G$ that, for each $1 \leq i \leq n$, maps $(i-1,i)$ homeomorphically to the interior of an edge, or if the graph $G$ is directed, a sequence of oriented edges $e_1, \dots, e_n$ such that $ter(e_i)=i(e_{i+1})$ for each $1 \leq i \leq n-1$. We may on occasion also allow for $e_1$ and $e_n$ to be partial edges.
Given any path $\gamma=e_1 \cdots e_n$, we will denote its initial vertex, i.e. $i(e_1)$, by $i(\gamma)$ and its terminal vertex, i.e. $ter(e_n)$ by $ter(\gamma)$.

A \emph{loop} $\al$ in $G$ is the image of an immersion $\al \colon \SS^1 \to G$. We will associate to each loop an edge-path unique up to cyclic ordering.

In a directed graph $G$, we will call a path \emph{directed} that either crosses all edges in a positive direction (a \emph{positive path}) or crosses all edges in a negative direction (a \emph{negative path}). 
The operation of path concatenation will be denoted by $*$.
\end{df}

\begin{df}[Illegal turns and gates]
Let $f \colon G \to H$ be a regular map. Let $e,e' \in E(G)$ be oriented edges with the same initial vertex. We call $\{e,e'\}$ a \emph{turn}. We say a turn $\{e,e'\}$ is an \emph{illegal turn} for $f$ if the first edge of the edge-path $f(e)$ equals the first edge of the path $f(e')$. The property of forming an illegal turn is an equivalence relation and the equivalence classes are called \emph{gates}. 
\end{df}

\begin{df}[Train track structures]
Let $f \from G \to H$ be a regular map. If every vertex of $G$ has $\geq 2$ gates, then we call the partition of the turns of $G$ into gates a \emph{train track structure} and say that $f$ induces a \emph{train track structure} on $G$. An immersed path $\al: I \to G$ will be considered \emph{legal} with respect to a given train track structure if it does not contain a subpath $e_ie_j$ where  $\{\overline{e_i}, e_j \}$ is an illegal turn.
\end{df}

\begin{rk}
The image of a legal path will be locally embedded.
\end{rk}

\begin{df}[Transition matrix]
The \emph{transition matrix} of a regular self-map $f \from G \to G$ is the square $|E(G)| \times |E(G)|$ matrix $(a_{ij})$ such that $a_{ij}$, for each $i$ and $j$, is the number of times $g(e_i)$ passes over $e_j$ in either direction.\footnote{This matrix is the transpose of the transition matrix as Bestvina and Handel define it in \cite{bh92}, but this definition will have a stronger relationship with the change-of-metric matrix we define later.}

We define the \emph{transition matrix} for an element $\Phi \in Aut(F_r)$ to be the transition matrix of $f_{\Phi}$ (see Definition \ref{d:OutAction}).
\end{df}

\subsection{Perron-Frobenius theory}{\label{s:PFdense}}

We are particularly interested in positive matrices (defined below) because of their known properties (due to Perron-Frobenius theory) of contracting, the \emph{positive cone} $\RR^d_{+}=\{ v\in \RR^d \mid v_i>0, ~ i=1,\dots d \}$.

\begin{df}[Positive matrices, Perron-Frobenius eigenvalues and eigenvectors]
We call a matrix $A=[a_{ij}]$ \emph{positive} if each entry of $A$ is strictly positive.
By Perron-Frobenius theory, we know that each such matrix has a unique eigenvalue of maximal modulus and that this eigenvalue is real.
This eigenvalue is called the \emph{Perron-Frobenius (PF) eigenvalue} of $A$. It has an associated eigenvector whose entries are each strictly positive. We call the eigenvector with strictly positive entries and such that all entries sum to one the \emph{Perron-Frobenius (PF) eigenvector}.
\end{df}

\begin{df}[Weak convergence]
A sequence $\{A_k\}_{k=1}^{\infty}$ of $d \times d$ matrices, restricted to vectors in $\RR^d_{+}$, \emph{converges weakly} if the sequence $\{A_k({\RR}^d_{+})\}_{k=1}^{\infty}$ converges projectively to a point.
\end{df}

\begin{rk}
Perron-Frobenius theory also tells us that, for a positive matrix $M$, the sequence $\{M^k \}_{i=1}^{\infty}$ weakly converges to the line spanned by the PF eigenvector.
\end{rk}

\subsection{The Lipschitz metric}{\label{SS:Lipschitz}}

\begin{df}[Difference in markings]
Let $x =(G,m,\ell)$ and $y = (G',m',\ell')$ be two points in $\os$. Denote by $h \colon G \to R_r$ a homotopy inverse of $m$. \emph{A difference in markings} is a linear map $f \colon G \to G'$ homotopic to $m' \circ h$.
\end{df}

\begin{df}[Stretch]
Let $\al$ be a conjugacy class in $F_r$, equipped with a free basis $A$. By abuse of notation we may think of $\al$ as an immersed loop $\al \colon \SS^1 \to R_r$ in $R_r$ via the identification of the edges of $R_r$ with $A$. For $x = (G,m,\ell) \in \os$, let $\al_x$ denote the unique immersed simplicial loop in $G$ homotopic to $m(\al)$.

Given a conjugacy class $\al$ in $F_r$ and $x \in \os$, we define $l(\al, x)$ as the length of $\al_x$. 
(Notice that since $\al_x$ is a simplicial loop in $x$, its length does not depend on the particular metric structure chosen for $x$, see Remark \ref{metricStructure}). For $x,y \in \os$ define the \emph{stretch} of $\al$ from $x$ to $y$ as $\st{\al}{x}{y} := \frac{l(\al,y)}{l(\al,x)}.$
\end{df}

The following theorem is attributed to either White or Thurston. It can be found in \cite{AlgomKfirAxes}.

\begin{thm}{\label{t:lipstretch}}
Given a continuous map $f$ of metric spaces, let $Lip(f)$ denote the Lipschitz constant for $f$. Then for each pair of points $x,y \in \os$, we have
\begin{equation}{\label{e:lipstretch}}
\inf \{ \lip{f} \mid f \colon x \to y \text{ is a difference in marking } \} =
\sup \{ \st{\al}{x}{y} \mid \al \in F_r\} \text{.}
\end{equation}
Moreover, both the infimum and supremum are realized.
\end{thm}

\begin{df}[Lipschitz metric]\label{d:LipMet}
The \emph{Lipschitz metric} $d(x,y)$ is defined as the log of either of the quantities in Equation \ref{e:lipstretch}.
This function is not symmetric but satisfies the other axioms of a metric \cite{b12}.
\end{df}

A difference in marking that achieves the minimum Lipschitz constant of (\ref{e:lipstretch}) is called an \emph{optimal map}. A loop that achieves the maximum stretch is called a \emph{witness}. 
For each $x,y \in \os$ there exist optimal maps and witnesses. 

\begin{remark}\label{LipschitzBalls}
An open ball based at $x$ with radius $R$ is either
\[ B_{\rightarrow}(x,R) = \{ y \in \os \mid d(x,y) <R \} \text{  or  } B_{\leftarrow}(x,R) = \{ y \in \os \mid d(y,x) <R \}. \]
In either case, the simplicial topology is finer (or equal) to the Lipschitz topology. 
\end{remark}

For a given difference of marking $f$, the subgraph of $G$ where the Lipschitz constant is achieved is called the \emph{tension graph}, usually denoted $\Delta_f$. Notice that $f$ induces a train track structure on $\Delta_f$.
Proposition \ref{witness} gives one way to identify witnesses. 

\begin{prop}\cite{bf11}{\label{witness}}
Let $x$ and $y$ be two points in $\os$, let $f \from x \to y$ be a map, and let $\Delta_f$ be the tension graph of $f$. 
If $\Delta_f$ contains a legal loop, then $f$ is an optimal map and any legal loop in $\Delta_f$ is a witness.  Conversely, if $\al \subset x$ is a witness, then it is a legal loop in $\Delta_f$.
\end{prop}

\begin{prop}\label{geodTestProp}
Let $f \colon x \to y$ and $g \colon y \to z$ be difference in markings maps. Let $\al$ be a conjugacy class in $F_r$ satisfying that $\al_x$ is $f$-legal and contained in $\Delta_f$ and that $\al_y$ is $g$-legal and contained in $\Delta_g$. Then
$d(x,z) = d(x,y) + d(y,z)$.
\end{prop}

\begin{proof}
Since $\al_x$ is contained in $\Delta_f$, the map $f$ stretches each edge of $\al_x$ by $\lambda_f = \lip{f}$. Moreover, since $\al_x$ is $f$-legal, $l(\al,y) = \lambda_f \cdot l(\al,x)$. Similarly, if $\lambda_g = \lip{g}$, then
$l(\al,z) = \lambda_g \cdot l(\al,y) = \lambda_f\lambda_g \cdot l(\al,x)$. 
Thus, $\st{\al}{x}{z} = \lambda_f\lambda_g$ and therefore $d(x,z) \geq \log(\lambda_f\lambda_g)$. By Proposition \ref{witness}, $\al_x$ and $\al_y$ are both witnesses, hence $d(x,y) = \st{\al}{x}{y} =  \log \lambda_f$ and $d(y,z) = \st{\al}{y}{z}= \log \lambda_g$. Thus, we have $d(x,z) \geq \log\lambda_f + \log \lambda_g =  d(x,y) + d(y,z)$. The triangle inequality gives us an equality.
\end{proof}

\section{Fold paths and geodesics}{\label{S:foldpaths}}

In this section we introduce fold lines and prove results that will allow us to construct Lipschitz geodesics from certain infinite sequences of nonnegative matrices (unfolding matrices).
\bigskip

\begin{df}[Unparametrized geodesic]\label{D:UnparametrizedGeodesic} Let $I \subset \RR$ be a generalized interval in $\RR$. An \emph{unparametrized geodesic} in $\os$ is a map $\Gamma \from I \to \os$ satisfying that: 
\begin{enumerate}
\item for each $s<r<t$, we have $d(\Gamma(r), \Gamma(t)) = d(\Gamma(r), \Gamma(s))+ d(\Gamma(s), \Gamma(t))$ and
\item there exists no nontrivial subinterval $I' \subset I$ and point $x_0 \in \os$ such that $\Gamma (t)=x_0$ for each $t \in I'$.
\end{enumerate}
\end{df}

\begin{remark}
If $\Gamma$ is an unparametrized geodesic then there exists a generalized interval $I'$ and a homeomorphism $h \from I' \to I$ so that $\Gamma \circ h$ is an honest directed geodesic, i.e. for all $s<t$, we have that $d(\Gamma \circ h(s) , \Gamma \circ h(t)) = t-s$.
\end{remark}

Again $A=\{X_1, \dots X_r\}$ will denote a fixed free basis of $F_r$.

\begin{df}[Fold automorphism]\label{FoldAutomorphism}
By a \emph{fold automorphism} we will mean a ``left Nielsen generator," i.e. an automorphism of the following form ($i \neq j$):
\begin{equation}{\label{foldsFromMatrix}}
\Phi_{ij}(X_k) = \left\{
\begin{array}{ll}
 X_jX_k & \text{for } k=i \\
 X_k & \text{for } k \neq i.
\end{array} \right.
\end{equation}
By notation abuse $\Phi_{ij}$ will also denote the map $R_r \to R_r$ that corresponds to the above automorphism after identifying its positively oriented edges with the basis $A$.
\end{df}

To a fold automorphism one can associate a matrix.

\begin{df}[(Un)folding matrix]\label{dfFoldingMatrices}
Let $i \neq j \in \{1,\dots , r\}$. Then the $(i,j)$ \emph{folding matrix} $T_{ij}$ has entries
$t_{kl}$, where

\begin{equation}\label{eqTmatrix}
t_{kl} = \left\{
\begin{array}{rl}
1 & \text{if } k = l\\
-1 & \text{if } (k,l)=(i,j) \\
0 & \text{otherwise}
\end{array}
\right.
\end{equation}

Notice that $T_{ij}$ is \underline{not} the transition matrix of $f_{ij}$, though it will relate to the change-of-metric matrix coming from a folding operation. 
The matrix $T_{ij}$ is invertible and we call $M_{ij} := T_{ij}^{-1}$ the $(i,j)$ \emph{unfolding matrix}. Notice that the entries of $M_{ij}$ are $m_{kl}$, where

\begin{equation}\label{eqTinverse}
m_{kl} = \left\{
\begin{array}{rl}
1 & \text{if } k = l \text{ or } (k,l) = (i,j)\\
0 & \text{otherwise.}
\end{array}
\right.
\end{equation}
Notice also that the nonnegative \emph{matrix} $M_{ij}$ is the transition matrix of $\Phi_{ij}$. We hence sometimes write $M( \Phi_{ij})$ for this matrix.
\end{df}

\begin{df}[A combinatorial fold]\label{CombinatorialFolds}
Let $G$ be a graph whose oriented edges are numbered. Let $(e_k, e_j)$ be a pair of distinct oriented edges with $i(e_k) = i(e_j)$. A combinatorial fold is a tuple $(G, (e_k, e_j), G', f)$ where $G'$ is a graph and $f \from G \to G'$ is a quotient map that identifies an initial segment of $e_k$ with an initial sement of $e_j$. 
\begin{enumerate}
\item when $f$ identifies part of $e_k$ with all of $e_j$ we will call it a \emph{proper full fold}. We will call this tuple ``folding $e_k$ over $e_j$." 
\item when $f$ identifies all of $e_k$ with all of $e_j$ we will call it a \emph{full fold} and say that ``$e_k$ and $e_j$ are fully folded."
\item when $f$ identifies a proper segments of $e_k$ and $e_j$ we will call it a \emph{partial fold} and say that ``$e_k$ and $e_j$ are partially folded."
\end{enumerate}
\textbf{Notation} We sometimes repress some of the data depending on the contex. We will denote the combinatorial fold by $f$, or  $G \rightharpoonup G'$, or $(G, (e_i, e_j))$ depending on which data we want to emphasize. 
\end{df}

\begin{df}[Direction matching folds]\label{DirectionMatchingFolds}
Let $G$ be an oriented graph. A combinatorial fold $(G,(e_i, e_j))$ is \emph{direction matching} if $e_1$ and $e_2$ are either both negative edges (we then call $f$ \emph{negative}) or both positive edges (we then call $f$ \emph{positive}).
\end{df}

\begin{obs}\label{directionMatchingIndOrientation}
Let $G$ be an oriented graph and $f \colon G \far H$ a direction matching combinatorial fold, then $f$ induces an orientation on the edges of $H$. Moreover, $f$ maps each positive edge of $G$ to a positive edge-path in $H$ (of simplicial length $\leq 2$). 
\end{obs}

\begin{df}[Allowable folds]\label{AllowableFolds}
Let $x_0 = (G,m,\ell)$ be a point in Outer Space. The full fold $(G,(e,e'))$ is \emph{allowable} in $x_0$ if the following two conditions hold:
\begin{enumerate}
\item $\ell(e) \geq \ell(e')$.
\item If $ter(e) = ter(e')$, then $\ell(e) > \ell(e')$. In this case this is a \emph{proper full fold}. 
\end{enumerate}

Suppose $G$ is a rose with its edges enumerated. Let $n_\tau \from \sig_{(G, m)}\to \calS_{|E|}$ be the induced homeomorphism. The set where $\Phi_{ij}$ is  allowable will be denoted by $\sig_{(G,m)}^{(i,j)}$. The target graph $G'$ is also a rose and $\Phi_{ij}$ induces an enumeration of the edges of $G'$ by declaring the edge $\Phi_{ij}(e_k)$ to be the $k^{th}$ edge for each $k \neq i$ and calling the remaining edge of $G'$ the $i^{th}$ edge. 
\end{df}

\begin{lem}\label{EnuLemma}
Let $\hat\sig_0$ be the unprojectivized base simplex, and let $\Phi$ be a fold automorphism. Assume that the edges of $R_r$ have been enumerated and let $n_0 \from \hat\sig_0 \to \calS_r$ be the induced identification. $\Phi \from \hat\sig_0 \to \hat\sig_1$ induces a new enumeration of the edges of the target rose and we get a new identification $n_1\from \hat\sig_1 \to \calS_r$. Then for each $x \in \hat\sig_1$ we have $n_1(x) = n_0(x \cdot \Phi^{-1})$ (as defined in Definition \ref{d:OutAction}). 
\end{lem}

If a fold $(G, (e,e'))$ is allowable at a point $x_0$ in Outer Space, one can construct a  path $\{\hat x_t\}_{t \in [0, 1]}$ in unprojectivized outer space, called a ``fold path." This is done by identifying initial segments of $e$ and $e'$ of length $t\ell_0(e')$, for $0 \leq t \leq 1$. The quotient map $f_{t,0} \from x_0 \to x_t$ is a homotopy equivalence, as are the quotient maps $f_{t,s}$ for $0 \leq s \leq t \leq 1$. By projectivizing we get a path $\{x_t\}$ in Outer Space. 

\begin{df}[Basic fold paths]\label{BasicFoldPath}
Given an allowable fold as above,
the path $ \mF \from [0,1] \to \os$ defined by
$t \mapsto x_t$ is the \emph{fold path in $\os$ starting at $x_0$ and defined by folding $e$ over $e'$}.
The path $t \to \hat x_t$ will be called the \emph{(unprojectivized) basic fold path}. We will not always use the hat notation when discussing unprojectivized paths but will mention whether the image lies in $\os$ or $\uos$, if it is otherwise unclear. 
\end{df}

\begin{df}[Change-of-metric matrix]
Let $G,G'$ be graphs and assume we have numbered their oriented edges. Let $\Psi$ be a linear map from a subset of  the unprojectivized simplex $\hat{\sig}_{(G, \mu)}$ to the unprojectivized simplex $\hat{\sig}_{(G', \mu')}$. Then $\Psi$ may be represented by an $|E(G')| \times |E(G)|$ matrix. This matrix will be called the \emph{change-of-metric} matrix.   
\end{df} 

\begin{lem}\label{roseChange}
Let $\Phi$ be an allowable fold automorphism on the point $x_0 = (R_r, m, \ell)$ and suppose the edges of $R_r$ have been numbered so that $\Phi = \Phi_{ij}$. Let $x_1$ be the folded graph and suppose $\hat\sig_0$ and $\hat\sig_1$ are unprojectivized open simplices containing respectively $x_0$ and $x_1$. 
The change-of-metric matrix for the fold operation from $n_0(\hat \sig_{(R,m)}^{(i,j)})$ to $n_1(\hat \sig_{(R,\Phi \circ m)})$ is  the matrix $T_{ij}$ of Definition \ref{dfFoldingMatrices}.
\end{lem}

\begin{df}[Fold paths]\label{FoldPath}
A \emph{fold path} $\mF \from [0,k] \to \os$ is a path in $\os$ that may be divided into a sequence of basic fold paths $\{\mathcal{F}_i \}_{i=1}^k$ as in Definition \ref{BasicFoldPath},
so that $\mF_i(1) = \mF_{i+1}(0)$ for each $i$.
As above, we will denote by $x_t = (G_t, m_t, \ell_t)$ the points of the path in $\os$. The maps $f_{t,s}$ for $s,t \in [0,k]$ will be defined similarly as above.
\end{df}

If $\mF$ is a fold path from $\mF(0) = x$ to $\mF(k) = y$ we sometimes denote it  by
$\mF\from x \far y$.

The following Lemma follows from Lemma \ref{roseChange} and its proof is left to the reader.

\begin{lem}\label{L:SequenceAllowable}
Let $f_1, \dots, f_k$ be a sequence of combinatorial proper full folds of the $r$-rose, with associated change-of-metric matrices $T_1, \dots, T_k$ having respective inverse matrices $M_1, \dots, M_k$. Suppose that $v \in M_1 \cdots M_k(\RR^r_+)$. 
Then the combinatorial fold $f_l$ is allowable in the metric graph $n_{l}^{-1}(T_l \cdots T_1(v))$, for each $2 \leq l \leq k$ (and any marking). 
Furthermore, applying $f_1, \dots, f_k$ to $x_0 \in \sig_0$ will result in the point $n_k^{-1}(T_{k} \cdots T_1(n_0(x_0)))$ of the simplex $\sig_{(R_r, f_{k} \circ \cdots \circ f_1 \circ m)}$.
\end{lem}

\begin{lemma}\label{l:DefiningMetric}
Let $\{D_i\}_{i=1}^\infty$ denote a sequence of nonnegative invertible matrices such that, for some $N \in \NN$, we have that $D_1 \dots D_i$ is strictly positive for all $i \geq N$. Then there exists a vector $w_0 \in \RR_+^r$ so that, if we define $w_{l+1} := D^{-1}_{l+1} w_l$, for each $l$, then each $w_l$ is a positive vector. 
 
\end{lemma}

\begin{proof}
Let $\ol{\RR_+^r}$ denote the set of vectors with nonnegative entries, and recall the map $q$ from Equation \ref{e:Projectionq}.
Let $M = D_1 \cdots D_N$, by the assumption it is a  nonnegative matrix. Consider $\mathcal{I} = \bigcap_{i=1}^\infty D_1 \dots D_i(\RR_+^r)$.
Note that 
\[ \mathcal{I} \supset \bigcap_{i=1}^\infty D_1 \dots D_i(M(\ol{\RR_+^r})) 
\supset \bigcap_{i=1}^\infty q \left(D_1 \dots D_i\left( M(\ol{\RR_+^r})\right) \right)  \] 
and the latter is nonempty since it is an intersection of nested compact sets. 
Moreover, $\mathcal{I} \subset M(\RR^r_+) \subset \RR^r_+$. 
Choose $w_0 \in \mathcal{I}$.
Then $w_0 \in \mathcal{I} \subset D_1 \cdots D_l(\RR_+^r)$ implies that $w_{l} = D^{-1}_{l} \cdots D^{-1}_1 w_0 \in \RR_+^r$ for each $1 \leq l < \infty$.
\end{proof}

\begin{cor}\label{CorInitialPoint}
Let $\{f_i\}_{i=0}^{\infty}$ be a sequence of combinatorial proper full folds of the $r$-rose, with associated change-of-metric matrices $\{T_i\}_{i=0}^{\infty}$ having respective inverse matrices $\{M_i\}_{i=0}^{\infty}$. 
If there exists some $N \in \NN$, so that for all $i>N$ the matrix $D_1 \dots D_i$ is strictly positive, then there exists a vector $w_0 \in \RR_+^r$
so that the infinite fold sequence $\{f_i\}_{i=0}^{\infty}$ is allowable in the rose $x_0 = n_0^{-1}(w_0)$.
\end{cor}

\begin{prop}\label{foldLineGeodesics}
Let $\{\mF_i \from x_i \far x_{i+1}\}_{i=0}^k$ be a sequence of fold paths with fold maps $\{f_{s,t}\}_{s\geq t\geq 0}$.
Suppose there is a conjugacy class $\al$ in $F_r$ satisfying that, for each $i$, the realization $\al_{x_i}$ of $\alpha$ in $x_i$ is legal with respect to the train track structure induced by $f_{i+1,i}$. Then the corresponding fold path $\im (\mF) = \{ x_t \}_{t \in [0,k]}$ is an unparametrized geodesic, i.e. for each $r\leq s \leq t$ in $[0,k]$, we have $d(x_r,x_t) = d(x_r, x_s) + d(x_s, x_t)$.
\end{prop}

\begin{proof}
The proof uses Propositions \ref{witness} and \ref{geodTestProp} and is left to the reader.
\end{proof}

Suppose $y_0$ is a graph and $f_0 \colon y_0 \far y_1$ is direction matching, then by Observation \ref{directionMatchingIndOrientation}, we have that $y_1$ inherits an orientation such that the image of each edge is a positive edge-path (see Definition \ref{defPaths}). 

\begin{cor}\label{PositiveFoldLineGeodesics}
Let $y_0$ be a directed metric $r$-rose graph with length vector $v$ and let $\{f_i \colon y_i \far y_{i+1} \}^{i=k}_{i=0}$ be an allowable sequence of proper full folds. Suppose that for each $i=0, \dots, k-1$ the fold $f_{i+1}$ is direction matching with respect to the orientation of $y_{i+1}$ inherited by $f_i$. For each $0 \leq i \leq k$, let $\{\mathcal{F}_i \colon y_i \far y_{i+1} \}$ denote the fold path determined by $\{f_i\}$. Then the corresponding fold path $Im (\mathcal{F})=\{y_t\}_{t \in [0,k]}$ is an unparametrized geodesic.
\end{cor}

\begin{proof}
Without generality loss we can assume that the marking on $R_r$ is the identity. Then, by Proposition \ref{foldLineGeodesics}, it suffices to show that there exists a conjugacy class $\alpha$ in $F_r$ satisfying that, for each $i$, the realization $\alpha_{y_i}$ of $\alpha$ in $y_i$ is legal with respect to the train track structure induced by the map $f_i$. We claim that this holds for the conjugacy class $\alpha$ of a positive generator $X_1$. By induction suppose that $\alpha_{y_i}$ is a positive loop. Since $f_i$ maps positive edges to positive edge-paths, $\alpha_{y_{i+1}}$ is also a positive loop.
Since $f_{i+1}$ is direction matching, $\al_{y_{i+1}}$ is $f_{i+1}$-legal and positive. 
\end{proof}

\section{Brun's algorithm and density of Perron-Frobenius eigenvectors}{\label{s:BrunsAlgorithm}}

We introduce fibered systems so that we can use an algorithm of Brun to prove, in Theorem \ref{T:ConvergenceInAngle}, that the sequence of Brun matrices converge weakly for a full measure set of points. These matrices are unfolding matrices, which will be significant for proving Theorems A, B, and C in the next sections.

\subsection{Fibered Systems}{\label{ss:FiberedSystems}}
The following definitions are taken from \cite{s00}.

\begin{df}[Fibered systems]{\label{d:FiberedSystems}}
A pair $(B,T)$ is called a \emph{fibered system} if $B$ is a set, $T \from B \to B$ is a map, and there exists a partition $\{ B(i) \mid i \in I\}$ of $B$ such that $I$ is countable and $T|_{B(i)}$ is injective. The sets $B(i)$ are called \emph{time-1 cylinders}.
\end{df}

\begin{df}[Time-$s$ cylinder]{\label{d:Cylinders}}
For each $x \in B$, one can define a sequence
$\Phi(x) = (i_1(x), i_2(x) , \dots ) \in I^{\NN}$ by letting
$i_s(x) = i \iff T^{s-1}x \in B(i).$ 
In other words, $i_s(x)$ tells us which cylinder $T^{s-1}x$ lands in. Then a \emph{time-$s$ cylinder} is a set of the form
\[ B(i_1, \dots ,  i_s) = \{ x \in B \mid i_1(x) = i_1,  \dots,  i_s(x) = i_s \}.\]
\end{df}

From the definitions we have $B(i_1, \dots , i_{s+1}) =  B(i_1, \dots , i_s)  \cap T^{-s} B(i_{s+1})$.

\subsection{The unordered Brun algorithm in the positive cone}{\label{SS:BrunUnordered}}

The following algorithm (commonly referred to as \emph{Brun's algorithm}) was introduced by Brun \cite{b57} as an analogue of the continued fractions expansion of a real number in dimensions 3 and 4. It was later extended to all dimensions by Schweiger in \cite{sch82}.

\begin{df}[Brun's (unordered) algorithm]\label{D:BrunUnordered} \emph{Brun's unordered algorithm} is the fibered system $(\pc{n},T)$ defined on \[ \pc{n} = \RR^n_+ = \{(x_1, \dots, x_n) \in \RR_+^n \mid x_i > 0 \text{ for all } 1 \leq i \leq n \}\] by
\begin{eqnarray*}
T \colon  \pc{n} & \to & \pc{n} \\
(x_1, \dots , x_n) & \mapsto & (x_1, \dots , x_{m-1}, x_m-x_s, x_{m+1}, \dots , x_n),
\end{eqnarray*} 
where 
$$m(x) := \min \left\{ i \left| x_i = \max_{1 \leq j \leq n} x_j  \right\}, \right.
\quad
s(x) := \min \left\{ i \neq m(x) \left| x_i = \max_{
1 \leq j\neq m(x) \leq n} x_j  \right\}. \right.$$
In other words, for each $x=(x_1, \dots, x_n) \in \pc{n}$, we have that $m(x)$ is the first index that achieves the maximum of the coordinates and $s(x)$ is the first index that achieves the maximum of all of the coordinates except for $m(x)$.
\end{df}

Notice that, letting $(i,j)=(m(x),s(x))$, the transformation $T$ is just left multiplication by the matrix $T_{i,j}$ from Definition \ref{dfFoldingMatrices}. Then, given a vector $v_0=(x_1, \dots, x_n) \in C^n$ with rationally independent coordinates, one obtains an infinite sequence $$\{v_k=(x_1^k, \dots, x_n^k)\}_{k=1}^\infty \subset C^n,$$ 
where $v_{k+1}$ is recursively defined by $v_{k+1}=T_{m(v_k),s(v_k)}v_k$.

\begin{df}[Brun sequence]
In light of the above, the sequence for Brun's unordered algorithm (which we call the \emph{Brun sequence}) will consist of the ordered pairs $\overrightarrow{k_s}=(i_s,j_s)$, where $(i_s,j_s)=(m(v_s),s(v_s))$. Further, the sequence $\Theta(x)= \{\overrightarrow{k_s}\}_{s=1}^\infty$ will determine a sequence of folding matrices, which we denote by
$\{ T_{s}^x \}_{s=1}^\infty$, where $T_s^{x} = T_{\overrightarrow{k_s}(x)}$.
We let $\{M_s^x\}_{s=1}^\infty$ denote the corresponding inverses, i.e. the unfolding matrices. We denote their finite products by
\begin{equation}\label{A_kDef}
A^{x}_s = M_1^{x} M_2^{x} \cdots  M_{s}^{x} \text{ for each } s \in \NN.
\end{equation}
\end{df}

\smallskip

\noindent Then, as above, for each $v_0 \in \pc{n}$, we have a sequence of vectors $\{v_s\}_{s=1}^{\infty} \subset \pc{n}$, where $ v_{s+1} = T v_s = T^{v_0}_{s+1} v_s$.
Thus, $v_s=M_{s+1}^{v_0} v_{s+1}$, and hence
$v_0 = A^{v_0}_s v_{s}$. So
\begin{equation}\label{eqIntOfCones}
v_0 \in \bigcap_{s=1}^{\infty} A_s^{v_0} ( \pc{n}).
\end{equation}
This fact will become particularly important in the proof of Theorem \ref{thm_PF_dense}.

\begin{df}[Brun matrix]
For each vector $v_0 \in C^r$, we call each matrix
\begin{equation}\label{eqA}
A_n = M^{v_0}_{i_1,j_1} \cdots M^{v_0}_{i_n,j_n}
\end{equation}
a \emph{Brun matrix}. We let $\B_r$ denote the set of $r$ x $r$ Brun matrices. 
\end{df}

\begin{rk}
When it is clear from the context, we may leave out reference to the starting vector $v_0$ and simply write $A_s$, $T_s$, $M_s$, etc.
\end{rk}

\begin{prop}\label{BecomesPos}
Let $\dss$ be the set of rationally independent vectors in $\pc{r}$, then for each $x \in \dss$ there exists an $N \in \NN$ so that $A_n^x$ is a positive matrix for each $n>N$.
\end{prop}

\begin{proof}
We fix $v_0$ and omit the index $v_0$ from the notation below.
We first prove that, for each $i$, and for each $h \in \NN$, there exists some $m>h$ and some $j$ such that $\overrightarrow{k_m} = (i,j)$.
Starting with $v_h$ and until $(v_m)_i$ becomes the largest coordinate, at each step one subtracts a number $\geq (v_m)_i=(v_h)_i$ from some coordinate $\geq (v_m)_i$.
This can only happen a finite number of times before each coordinate apart from $(v_m)_i$ becomes less than $(v_m)_i = (v_h)_i$.

\smallskip

Consider $A_n$ as in (\ref{A_kDef}). To prove the proposition, it suffices to show that, for each
$(i,j)$, there exists a large enough $N$ so that for all $n>N$ the $(i,j)$-th entry of $A_n$ is positive.
In fact, it is enough to show that this entry is positive for some $A_n$.
Indeed $A_{n+1}$ is obtained from $A_n$ by adding one of its columns to another one of its
columns, so that an entry positive in $A_n$, will be positive in
$A_{n+1}$.

Fix $i, j$.
Let 
\[ \begin{array}{lcr}
a = \min\{ \text{ } t \text{ } \vert \text{ } \overrightarrow{k_{t}} = (i, c) \text{ for some } c\},  \text{ } 
b = \min\{ \text{ } t>a \text{ } \vert \text{ } \overrightarrow{k_{t}} = (j, d) \text{ for some } d\}.
\end{array} \]

Let $c_1$ be such that $\overrightarrow{k_{a}} = (i, c_1)$.
Observe that, since $c_1$ is the second largest coordinate in $v_{a}$, in the next vector $v_{a+1}$, either $i$ is still the largest coordinate or $c_1$ becomes the largest coordinate. 
There is some $N_1\geq 1$ and some index $c_2$ so that
\[ A_{a+N_1}= A_{a-1} M_{(i,c_1)}^{N_1} M_{(c_1,c_2)}.\]
We continue in this way,
$A_{n} = A_{a-1} M_{(i,c_1)}^{N_1} M_{(c_1,c_2)}^{N_2} M_{(c_2,c_3)}^{N_3} \cdots M_{(c_{t},c_{t+1})}^{N_t}$.
When $n=b$, we have that $c_t = j$. Thus, for $n=b-1$, we have
\[ A_{b-1} = A_{a-1} M_{(i,c_1)}^{N_1} M_{(c_1,c_2)}^{N_2} M_{(c_2,c_3)}^{N_3} \cdots M_{(c_{t},j)}^{N_t}.\]
It is elementary to see that the $(i,j)$-th entry of $M_{(i,c_1)}^{N_1} M_{(c_1,c_2)}^{N_2} M_{(c_2,c_3)}^{N_3} \cdots M_{(c_{t},j)}^{N_t}$ is positive.
This implies that the $(i,j)$-th entry of $A_{b-1}$ is positive.
\end{proof}

\subsection{Other versions of Brun's algorithm}{\label{SS:BrunOther}}

To use the results of Schweiger's books, we must give two other different, but related, versions of Brun's algorithm.

\begin{df}[Brun's ordered algorithm]\label{D:BrunOrdered} \emph{Brun's ordered algorithm} is the fibered system $(\Delta^{n},T')$ defined on
$\Delta^n := \{x \in \pc{n} \mid x_1 \geq \dots \geq x_n  \}$ by
\[ \begin{array}{rl}
T' \colon \Delta^n & \to \Delta^n \\
(x_1, \dots , x_n) & \mapsto (x_2, \dots, x_{i-1} , x_1 - x_2 , x_{i} , \dots , x_n),
\end{array}\]
 where $i=i(x) \geq 2$ is the first index so that $x_1-x_2 \geq x_{i}$ and, if there is no such index, we let $i(x) = n$.
The time-1 cylinders are $\Delta^n(i) = \{ x \mid x_{i-1}> x_1-x_2 \geq x_i\}$.
\end{df}

Notice that the transformation $T'$ is just left multiplication by the matrix $T_{1,2}$, followed by a permutation matrix that we denote $P_i$, 
determined by the cylinder $\Delta^n(i)$. (We denote $P_iT_{1,2}$ by $T_i'$ and its inverse by $M_i'$.) Hence, given a sequence of indices $\omega(x) = (i_1, i_2, \dots)$ with $1 \leq i_j \leq n$ for each $i_j$, one obtains a sequence 
of matrices $\{T_{i_j}'\}_{j=1}^{\infty}$. Given $v_0 \in \Delta^n(i_1, \dots, i_m)$, this gives a sequence of points $v_0, \dots, v_m \in \Delta^n$ such that $v_{k+1}=T_{i_j}'v_k$ for each $1 \leq k \leq m-1$. The fibered system 
sequence here will be $\omega(x)=(i_1, i_2, \dots)$ when 
$$x \in \bigcap_{m=1}^{\infty} \Delta^n(i_1, i_2, \dots, i_m).$$ 
If $\omega(x)=(i_1, i_2, \dots)$, we define
\begin{equation}
A^x_k=M_{i_1}' \cdots M_{i_k}'
\end{equation}
for each $k \in \NN$.

\begin{df}[Brun's homogeneous algorithm]\label{D:BrunHomogeneous} \emph{Brun's homogeneous algorithm} is the fibered system $(B_n,\overline{ T'})$ defined on 
\[ B_n = \{ x \in \RR^n \mid 1 \geq x_1 \geq x_2 \geq \dots \geq x_n  \geq 0 \}\] and where $\ol{T'} \from B_n \to B_n$ is such that the following diagram commutes:
\[ \xymatrix{
\Delta^{n+1} \ar[d]^p \ar[r]^{T'} &\Delta^{n+1} \ar[d]^p\\
B_n \ar[r]^{\overline{T'}} &B_n} \]
where $p\from \Delta^{n+1} - \{0\} \to B_n$ is defined by
\begin{equation}
p(x_1, \dots , x_{n+1}) = \left(
\frac{x_2}{x_1}, \dots , \frac{x_{n+1}}{x_1} \right).
\end{equation}

We denote the time-1 cylinders by $B_n(i)$.
\end{df}

\subsection{Relating the algorithms}{\label{SS:BrunRelations}}

\begin{df}
Let $\mathcal{O}\colon C^n \to \Delta^n$ be defined by 
$$(x_1, \dots, x_n) \mapsto (x_{i_1}, \dots, x_{i_n})$$ where 
$x_{i_1} \geq x_{i_2} \geq \dots \geq x_{i_n}$. Note that for a particular $x$, $\calO(x)$ is a permutation.
\end{df}

\begin{lem}\label{RelatingAlgorithms}
For each $x \in C^n$ and for each $m \in \NN$, there exist permutation matrices 
$P_{i_1}, P_{i_2}$ so that 
\[ A^x_m = P_{i_1} (A^{\mathcal{O}(x)}_m)' P_{i_2}. \]
\end{lem}

\begin{proof} This follows from the following commutative diagram:
\[ \xymatrix{
C^r \ar[d]^{\mathcal{O}} \ar[r]^{T} &C^r \ar[d]^{\mathcal{O}}\\
\Delta^r \ar[r]^{T'} &\Delta^r}. \]
\end{proof}

\begin{cor}\label{Positive}
For each $x \in C^n$ and for each $m \in \NN$, 
we have that $A^x_m$ is a positive matrix if and only if $ (A^{\mathcal{O}(x)}_m)'$ is a positive matrix.
\end{cor}

\begin{cor}\label{OrdPos}
For each irrational $x \in \Delta^n$ there exists an $N$ so that $(A^x_n)'$ is positive for all $n > N$.  
\end{cor}
\begin{proof}
This follows from Proposition \ref{BecomesPos} and Corollary \ref{Positive}. 
\end{proof}

\subsection{Weak convergence and consequences}{\label{SS:Convergence}}

Recall the definitions of the $(r-1)$-dimensional simplex $\mS_r$ in Definition \ref{simplexDefn} and the projection map $q \from \pc{r} \to \mS_r$ of Equation \ref{e:Projectionq}. We will show that for each $r \geq 2$, the set of Perron-Frobenius eigenvectors for the transition matrices of positive automorphisms in $Aut(F_{r})$ is dense in the simplex $\mS_{r}$.

It is proved in \cite[Theorem 21, pg. 5]{s00} that Brun's ordered algorithm is ergodic, conservative, and admits an absolutely continuous invariant measure. The proof uses R\'{e}nyi's condition, which further says that the measure is equivalent to Lebesgue measure. We are very much indebted to Jon Chaika for pointing out to us that we could use the ergodicity of Brun's algorithm  to prove the following theorem. 

\begin{thm}\label{T:ConvergenceInAngle}
There exists a set $K \subset \mS_r$ of full Lebesgue measure such that for each $x \in K$
\begin{equation}\label{eqConvInAngle} 
\bigcap_{j=1}^\infty A^x_j(\RR^r_+) = \span_{\RR_+}\{x\}. 
\end{equation}
\end{thm}

\begin{rk}
Before proceeding with the proof, we explain what we saw as an impediment to proving that the PF eigenvectors are dense in a simplex. It is possible to have a sequence of invertible positive integer $d$ x $d$ matrices $\{M_i\}_{i=1}^{\infty}$ so that 
$$\bigcap_{k=1}^{\infty} M_1 \cdots M_k (\RR^d_{+})$$ 
is more than just a single ray. The existence of such sequences of matrices was proved in the context of non-uniquely ergodic interval exchange transformations. There are several papers on the subject (including \cite{kn76}, \cite{k77}, \cite{v82}, \cite{m82}). Because it may not be straightforward to the reader outside of the field, we briefly explain how \cite{k77} implies the existence of such a sequence. 

We consider a sequence of pairs of positive integers $\{(m_k, n_k)\}_{k=1}^{\infty}$ satisfying the conditions of \cite[Theorem 5]{k77}. We look at
$$\bigcap_{k=1}^{\infty} A_{m_1,n_1}A_{m_2,n_2} \cdots A_{m_k,n_k} (\RR^d_{+}),$$ 
as defined on pg. 191. Keane explains on pg. 191 that the product of any two successive $A_{m_i,n_i}$ is strictly positive and that always $det(A_{m_k,n_k})=1$. We let $B_k=A_{m_1,n_1}A_{m_2,n_2} \cdots A_{m_k,n_k}$. \cite{k77} projectivizes $B_k$ to $\tilde B_k$, so that $\tilde B_k$ is a map of the 3-dimensional simplex $\mathcal{S}_4$. By Lemma 4, $\{\tilde{B}_k(0,1,0,0)\}_{k=1}^\infty$ is a sequence of vectors converging to a vector whose $2^{nd}$ entry is at least $\frac{1}{3}$. By Lemma 3 (when Theorem 5(ii) holds), $\{\tilde{B}_k(0,0,1,0)\}_{k=1}^\infty$ is a sequence of vectors converging to a vector whose $3^{rd}$ entry is at least $\frac{7}{10}$. But $\frac{1}{3} + \frac{7}{10} >1$ and we have assumed that we are in $\mathcal{S}_4$. So these limits must be distinct vectors.
\end{rk}

\begin{proof}[Proof of Theorem \ref{T:ConvergenceInAngle}]
Choose any $N$-cylinder $\Omega := \Delta^r(i_1, \dots i_N)$ such that the corresponding matrix $Z:=A_N'$ is positive (see Corollary \ref{OrdPos}). Let $\oOmega: = p(\Omega)$. Then $\mu(\oOmega)>0$ where $\mu$ is the Lebesgue measure on $B_r$. 
Since $\oT': B_r \to B_r$ is ergodic with respect to the Lebesgue measure, by Birkhoff's Theorem, there exists a set $\oK \subset B_r$ such that $\mu(\oK)=1$ and so that for each $\bar x \in \oK$ the following set is infinite:
\[ J(\bar x) := \{ n \in \NN \mid \oT'^n(\bar x) \in \oOmega \}. \]
 We let $K' := p^{-1}(\oK)$. Then for each $x' \in K'$ the set $I(x') := \{ n \in \NN \mid (T')^n(x') \in \Omega \}$ is infinite, as $n \in I(x')$ if and only if $n \in J(p(x'))$.

Let $K'' := \mathcal{O}^{-1}(K') \subset \RR^r_+$. If $x \in K''$, then $\mathcal{O}(x) \in K'$, and hence for each $n \in I(\mathcal{O}(x))$ we have
\[ (T')^n(\mathcal{O}(x)) \in \Omega. \]
Let $s \in \NN$ be arbitrary. Consider the first $s$ integers $\{j_1, \dots, j_s\}$ in $I(\mathcal{O}(x))$ satisfying that any difference between two of these numbers is $>N$ (where $N$ came from the original $N$-cylinder we started with). 
Let $N_1 \geq j_s + N + 1$. Then, for each $m>N_1$,
\[ (A^{\mathcal{O}(x)}_m)' = D_1 \cdots D_{j_1-1} Z D_{j_1+N} \cdots D_{j_2-1}  \cdots D_{j_s-1} Z D_{j_s+N} \cdots D_m , \]
where $Z$ is the positive matrix that we started with and the $D_i$ are the $M'_i$'s of Brun's ordered algorithm.
The matrix $Z$ appears in this product $s$ times.  
By Lemma \ref{RelatingAlgorithms}, for each $x \in K''$ and each $m \in \NN$, there exist permutation matrices $P_{i_1}, P_{i_2}$ so that 
\[ A^x_m = P_{i_1} (A^{\mathcal{O}(x)}_m)' P_{i_2}. \]
Hence, for this arbitrary $s$ we have found an $N_1(s) \in \NN$ so that for all $m>N_1$, 
\[ A^{x}_m = P_1 D_1 \cdots D_{j_1-1} Z D_{j_1+N} \cdots D_{j_2-1}  \cdots D_{j_s-1} Z D_{j_s+N} \cdots D_m P_2, \]
where $Z$ appears in this product at least $s$ times and the other matrices in this product are all invertible nonnegative integer matrices. 
Then, by ~\cite[Corollary 7.9]{f09}, Equation \ref{eqConvInAngle} holds true for each $x \in K''$. 

Let $q:\RR^r_+ \to \mS_r$ be the projection to the simplex in the positive cone (see (\ref{e:Projectionq})) and let $\mu'$ be the Lebesgue measure on $\mS_r$. Denote $K= q(K'')$. Then, since $\mu(\oK)=1$, we arrive at $\mu'(K)= 1$, as desired. 
\end{proof}

\begin{df}\label{PFPos}
We let $\mP_r$ be defined as:
$$\mP_r  :=\{ v_{PF} \in \mS_r \mid v_{PF}  \text{ is the PF eigenvector for some positive Brun matrix } M \in \B_r\}.$$ 
\end{df}

\begin{df}
Suppose 
\begin{equation}\label{eqA1}
A=M_{i_1,j_1} \cdots M_{i_n,j_n} \in \B_r.
\end{equation}
Then each $M_{i_k,j_k}$ is an unfolding matrix as in (\ref{eqTinverse}) and we can associate to it the fold-automorphism $f_{i_kj_k}$ (see Definition \ref{FoldAutomorphism}). Notice that $M_{i_k,j_k}$ is in fact the transition matrix for $f_{i_kj_k}$. To each $A \in \B_r$ as in (\ref{eqA1}) we associate the automorphism
\begin{equation}\label{mautoFoldDecomp}
\mauto_A = f_{i_n,j_n} \circ \dots \circ f_{i_1,j_1},
\end{equation}
whose transition matrix is $A$. We also call this automorphism $\mauto_v$, where $v$ is the PF eigenvector of $A$.
\end{df}

\begin{mainthmD}\label{thm_PF_dense}
For each $r \geq 2$, let $S_{l_1}^r$ be the set of unit vectors according to the $l_1$ metric in $\RR^r_+$. The set of Perron-Frobenius eigenvectors of matrices arising as the transition matrices of a positive automorphisms in $\text{Aut}(F_r)$ is dense in the $S_{l_1}^r$.
\end{mainthmD}

\begin{proof}
It will suffice to shows that $\mP_r$ is dense. By Theorem \ref{T:ConvergenceInAngle}, we know that Brun's algorithm converges in angle on a dense set of points $K \subset \mS_r$. Thus, given any $x \in \mS_r$ and $\eps > 0$, there exists some $x' \in B(x, \frac{\eps}{2}) \cap K$ on which Brun's algorithm weakly converges. Hence, there exists some $N$ such that, for each $n \geq N$, we have that $q(A^{x'}_nC^r) \subset B(x', \frac{\eps}{2})$. By possibly replacing $N$ with a larger integer, we can further assume that the $A^{x'}_n$ are positive (see Proposition \ref{BecomesPos}), so have PF eigenvectors. And the PF eigenvector $v_n$ for each $A^{x'}_n$ is contained in $A^{x'}_nC^r$ and hence is in $B(x', \frac{\eps}{2})$. Hence, there exists a $v_i \in\mP_r$ such that $d(v_i, x) < \eps$. 
\end{proof}

\section{Dense Geodesics in theta complexes}{\label{S:dg}}

In this section we construct a geodesic ray dense in the theta subcomplex whose top-dimensional simplex has underlying graph as in the right-hand graph of Figure \ref{thetaGraph1}. One could consider this a warm-up to the proof of Theorem B or interesting in its own right, as this is the minimal subcomplex containing the projection of the Cayley graph for $\out$.

\subsection{Construction of the fold ray}{\label{ss:SnakePath}}
We enumerate the vectors in $\mP_r$ from Definition \ref{PFPos} as $\{v_i\}_{i=1}^\infty$. For each $i$ there exists a positive matrix $A_{v_i}$ in $\B_r$ so that $v_i$ is the PF eigenvector of $A_{v_i}$. Further, there exists an automorphism $\mauto_{v_i}$ in $\A_r$ corresponding to $A_{v_i}$ (Equation \ref{mautoFoldDecomp}).
We also enumerate all possible fold automorphisms, as in (\ref{foldsFromMatrix}), by $h_1, \dots, h_{r(r-1)}$. 

We construct a sequence that contains each $\mauto_{v_i}^k \circ h_j$ with $i,k,j \in \NN$:
\begin{equation}{\label{e:snake}}
\mauto_{v_1} \circ h_1,~
\mauto_{v_2} \circ h_1,~ 
\mauto_{v_1}^2 \circ h_2,~
\mauto_{v_1}^3 \circ h_3,~
\mauto_{v_2}^2 \circ h_2,~
\mauto_{v_3} \circ h_1,~
\mauto_{v_4} \circ h_1 \dots
\end{equation}

Decompose each $\mauto_{v_i}$ in (\ref{e:snake}) according to (\ref{mautoFoldDecomp}), to obtain an infinite sequence of fold automorphisms 
$\{\Phi_k\}_{k=1}^{\infty}~\refstepcounter{equation}(\theequation)\label{f_auto_of_ray2}$.
Orienting $R_r$ and identifying its positive edges with the basis, $\Phi_1$ can be represented by the combinatorial fold $f_1$. The new graph $R_r$ inherits an orientation and an enumeration of edges and thus we may continue inductively to define the combinatorial fold $f_{k} 
\from G_{k-1} \to G_{k}$. The enumeration of edges induces a homeomorphism $n_k \from \hat \sig_{(G_k, f_k \circ \dots \circ f_1)} \to \mS_r$.

\begin{lem}\label{BasePoint}
Let $\rsimp$ be the base simplex, then 
there exists a point $x_0 \in \rsimp$ so that, for each $k \geq 0$, the fold $f_{k+1}$ is allowable in the folded rose after performing the sequence of folds $f_1, \dots , f_k$.
Moreover, this folded rose is
$x_k = n_k^{-1}(T_k \circ \dots \circ T_1 (n_0(x_0)))$ in the simplex $\sig_{(G, f_k \circ \dots \circ f_1 \circ m)}$.
\end{lem}

\begin{proof} 
This follows from Corollary \ref{CorInitialPoint}. Note that the positivity condition follows from the positivity of the matrices $A_{v_i}$.
\end{proof}

\smallskip

\begin{rk}
~\cite[Corollary 7.9]{f09} implies the metric on $x_0$ is unique, as the same positive matrix occurs in infinitely many of the products $M_{i_1} \cdots M_{i_k}$.
\end{rk}

\begin{df}[$\mR$]{\label{d:RayR}}
We let $\mR$ denote the infinite fold ray (Definition \ref{FoldPath})  initiating at $x_0$ and defined by the sequence of folds $\{ f_i\}$ as constructed above.
\end{df}
 
\begin{mainthmC}
For each $r \geq 2$, there exists a fold line in $\mathcal{T}_r$ that projects to a Lipschitz geodesic fold ray in $\mathcal{T}_r/\out$.
\end{mainthmC}

\begin{proof}
We recall $\mR$ from Definition \ref{d:RayR}. It is clear that $\mR$ is contained in $\calT_r$. $\mR$ is a geodesic ray by Corollary \ref{PositiveFoldLineGeodesics}.

First recall that the simplicial and Lipschitz metrics on $\os$ induce the same topology on $\os$. Hence, it suffices to prove density in the simplicial metric.

Let $\bar a \in \mathcal{M}_r$ and let $\eps>0$ be arbitrary. Lift $\bar a$ to a point $a \in \os$ in the interior of a top dimensional simplex $\tau$.
Let $y \in \tau$ be a point such that $d_s(a,y)<\eps$, and so that its coordinates are rationally independent. The point $y$ lies on a fold line $\calF_{i,j}$ from a point $x$ on one face of $\tau$ to a point $z$ in another face. Without generality loss assume $z \in \sig_0$, the base simplex. 
Moreover, without generality loss assume the combinatorial fold is $f_{1,2}$. 
Let $e_1' , \dots , e_r'$ denote the edges of $R_r$ as numbered in $\sig_0$. Enumerate the edges of $G_\tau$, the underlying graph of $\tau$, as $e_1, \dots e_{r+1}$, so that if  $c \from G \to R$ is the map collapsing $e_2$ in $G$, then $c(e_1) =e'_1$ and $c(e_i) = e_{i-1}'$ for each $i \geq 3$. 
We parameterize the path $[x,z]$ as an unfolding path $\hat \gamma_{1,2}(z,t)$ in unprojectivized Outer Space as follows. Let $\zeta = n_0(z) \in \calS_r$, then 
\[ n_{\tau}(\hat \gamma_{1,2}(z,t)) = (\zeta_1+t, t, \zeta_2-t, \zeta_3, \dots, \zeta_r).   \]
We note that $z = \hat \gamma_{1,2}(z,0)$, $x = \frac{1}{1-\zeta_2} \hat \gamma_{1,2}(z,\zeta_2)$, and for some $0<t_0 \leq \zeta_2$ we have $y = \frac{1}{1-t_0}\hat \gamma_{1,2}(z,t_0)$. 
Since the function $\gamma_{1,2}(z,t) = \frac{1}{1-t}\gamma_{1,2}(z,t)$ is continuous when $t$ is bounded away from 1, there exists an $\eps'>0$ so that for each $w \in B(z, \eps')$, the path $\gamma_{1,2}(w,t)$ passes through  $B(a,\eps)$. 

Now, by the density of PF eigenvectors (Theorem $D$), there exists a vector $v_i \in \mP_r$ such that $v_i \in n_0(B(z,\eps'))$. Let $\eps''>0$ be such that $n_0(B(z, \eps'))$ contains $B(v_i, \eps'')$. 
Let $K$ be large enough so that $q(A_{v_i}^k (\RR^{r}_+)) \subset B(v_i,\eps'')$ for all $k>K$.

Let $\Psi_k$ be the composition of $\Phi_1, \Phi_2, \dots$ from (\ref{f_auto_of_ray2}) up to the first fold automorphism in the decomposition of $g_{v_i}^k$. 
Choose $k>K$ so that the last fold automorphism in $\Psi_k$ is $\Phi_{1,2}$. 
Let $z_k$ be the rose-point in $\mR$ directly after performing the fold sequence of $\Psi_k$. Let $x_k$ be the rose-point directly before $z_k$. 

We claim that $\mR \cdot \Psi_k^{-1}$ is $\eps$-close to $a$. To see this recall that directly after $z_k$ in $\mR$ we perform the folds corresponding to $g_{v_i}^k$. Let $w_k$ be the rose-point in $\mR$ directly after these folds and let $n_s, n_m$ be the appropriate idenitifications of the simplices containing $z_k$ and $w_k$ with $\mS_r$. Then $n_s(z_k) = A_{v_i}^k(n_m(w_k))$. Hence $n_s(z_k)$ is $\eps''$-close to $v_i$. Thus  $n_s(z_k)$ is $\eps'$-close to $n_0(z)$. 

The point $z_k$ is in $\sig_0 \cdot \Psi_k$. Hence, by Lemma \ref{EnuLemma}, $n_0(z_k \cdot \Psi_k^{-1}) = n_s(z_k)$. Thus the fold path $[x_k, z_k]\cdot \Psi_k^{-1}$, which is a path induced by the fold $\Phi_{1,2}$, satisfies that its endpoint, $z_k\cdot \Psi_k^{-1}$, is $\eps'$-close to $z$. Thus, this fold line intersects $B(a,\eps)$, as desired. 

\end{proof}


\section{Finding a rose-to-graph fold path terminating at a given point}\label{S:dgAll}

This section is the first step in our expansion of our methods of Section \ref{S:dg} to obtain a dense geodesic ray in the full quotient of reduced Outer Space. 

In this section and in the next one, we find, for a dense set of points $y$ in reduced Outer Space, roses $x,z$ so that $y \in [x,z]$ and the difference in marking map $x \to z$ is positive. The path $[x,z]$ will be called a positive rose-to-rose fold line. 
It will replace our basic fold lines $\mathcal{F}_{i,j}$ in the proof of Theorem C. 

A rose-to-rose fold path will have two parts: a rose-to-graph part $[x,y]$ and a graph-to-rose part $[y,z]$. We begin in (Subsection \ref{graphDecompToLoops}) with decomposing the graph of $y$ into a union of positive loops (Lemma \ref{OrientEdges}). This allows us to find the rose $x$. 

\subsection{Decomposing a top graph into a union of directed loops}\label{graphDecompToLoops}

\begin{df}[Paths and distance in trees]
Let $T$ be a tree. Then for each pair of points $p,q$ in $T$ there is a unique (up to parametrization) path from $p$ to $q$. We denote its image by $[p,q]_T$ and, when there is no chance for confusion, we drop the subscript $T$. Given a tree $T$, let $d_T(\param, \param)$ denote the distance in $T$.
\end{df}

\begin{df}[Rooted trees]
A \emph{rooted tree} is a tree $T$ with a preferred vertex $v_0$ called a \emph{root}. 
A rooted tree can be thought of as a finite set with a partial order that has a minimal element - which is the root.
We will refer to the partial ordering induced by the pair $(T,v_0)$ as $\leq_T$, i.e. $w \leq_T w' \implies w \in[v_0,w']_T$.
\end{df}

\begin{remark}
In the figures to follow the root will always appear at the bottom.
\end{remark}

We use special spanning trees to guide us in finding the loop decomposition of $G$:

\begin{df}[Good tree]
Let $T$ be a rooted tree in $G$ and $e=(v,w)$ an edge. We call $e$ \emph{bad} if $v \nleq_T w$ and $w \nleq_T v$. Let $B(T)$ be the number of bad edges in $G$ with respect to $T$. When $B(T)=0$ we call $T$ \emph{good} (sometimes elsewhere called \emph{normal}).
\end{df}

We prove a somewhat stronger version of ~\cite[Proposition 1.5.6]{Diestel}.

\begin{prop}\label{goodTree}
For each $G \in \os$, and for each $E \in E(G)$ that is not a loop, there exists a rooted spanning tree $(T, v_0)$ so that $B(T)=0$ and $E \in E(T)$ and $v_0 = ter(E)$. Moreover, when $G$ is trivalent with no separating edges, $T$ can be chosen so that $deg_T(v_0)=1$.
\end{prop}

\begin{proof}
For an edge $e=(v,w)$ in $G$,
the union $[v_0,v]_T\cup[v_0,w]_T$ forms a tripod. Denote the middle vertex of this tripod by $q_e$,
i.e. $q_e$ satisfies $[v_0, q_e]_T = [v_0,v]_T \cap [v_0,w]_T$.
Let $d(e) = d(v_0, q_e)$.
We define the complexity $C(T) = (n(T),m(T))$ of $T$ by defining $n(T)$ and $m(T)$ as follows: 
\begin{equation*}
n(T) = \left\{
\begin{array}{ll}
- \min\{d(e) \mid e\text{ is bad} \} & \text{if } B(T) \neq 0\\
-\infty & \text{otherwise}
\end{array}
\right.
\end{equation*}
and let 
\[ m(T) = \# \{ e \in G \mid  e \text{ is bad and } d(e) = -n(T)\}. \] 
In particular, $m(T)=0$ if there are no bad edges, so that, if there are no bad edges, $C(T) = (- \infty, 0)$.

The complexity is ordered by the lexicographical ordering of the pairs.
Note that for $T$ bad, the complexity is bounded from below.
Indeed, when $G \in \os$, we have $|V(G)| \leq 2r-2$, so $n(T) \geq 2-2r$, hence $C(T) \geq (2-2r,1)$.
We will show that we can always decrease the complexity, so that $T$ can no longer have a bad edge.

Given an edge $E$ in $G$, we let $v_0 = ter(E)$. Let $G_1, \dots , G_N$ denote the components of $G-\{v_0\}$ with $\{v_0\}$ added back to each component separately. Thus $G= \cup_{i=1}^N G_i/\{v_0\}$. We will construct a good tree $T_i$ (rooted at $v_0$) in each $G_i$. Then $T = \cup_{i=1}^N T_i$ will also be a good tree rooted at $v_0$, since all edges $e \in E(G-T)$ have endpoints inside some $G_i$. 

Let $G_1$ be the component
containing the edge $E$. For $i \neq 1$ we choose a spanning tree $T_i$ in $G_i$ arbitrarily. For $i=1$ we construct a spanning tree $T_1$ in $G_1$ such that $deg(v_0)_{T_1}=1$: Denote by $v_1$ the endpoint of $E$ distinct from $v_0$. We define $T_1$ to be the union of $E$ with a spanning tree in $G_1-E$. 

We will modify each $T_i$ separately to make it a good tree in $G_i$. 
Suppose $e=(v,w)$  is a bad edge in $G_i$ realizing the minimal distance $-n(T_i)$.
Let $e_1$ be the first edge of $[v,q_e]_{T_i}$, and let $e_2$ be the last edge of $[v,q_e]_{T_i}$.
Let $T_i' = T_i \cup \{e\} \setminus \{e_2\}$.
We claim that $e_2 \neq E$. If $G_i \neq G_1$, then this is obvious. Otherwise, $deg_{T_1}(v_0)=1$ and $deg(q_e) \geq 2$, so that $v_0 \neq q_e$ and hence $e_2 \neq E$. Therefore, $E \in E(T_1')$ after the move and still $deg_{T_1'}(v_0)=1$. 
Next, notice that some bad edges of $T_i$ have become good in $T_i'$, for example $e$ is no longer bad, as is any edge from a vertex in $[v,q_e]_{T_i}$ to a vertex in $[q_e,w]_{T_i}$. 
Some bad edges remain bad. But the only edges that were good and became bad are edges with one endpoint in $[v,q_e]_{T_i}$ and one endpoint in a component of $T_i \setminus \{v\}$ that does not contain $e_1$ or $w$. For such an edge $f$, we have that $q_f = v$, so $d(f) = d(v, v_0)_{T_i'} > d(q_e, v_0)_{T_i'} = d(e)$ and the complexity has decreased.

If $G$ is trivalent with no separating edges, then no edge is a loop. This implies that there are three edges $E,E', E''$ incident at $v_0$. If the tripod $E\cup E' \cup E''$, was separating then each of its edges would be separating - a contradiction. Therefore, $G=G_1$ in this case and the proof is complete.
\end{proof}

\begin{lem}\label{OrientEdges}
If $G$ is a trivalent graph with no separating edges and $\{E,E'\}$ is a turn at the vertex $v_0$, then there exists an orientation on the edges of $G$ so that,
\begin{enumerate}
\item $G = \cup_{i=1}^r \al_i$ where each $\al_i$ is a positive embedded loop, 
\item $\al_i \cap ( \cup_{j=1}^{i-1} \al_j )$ is a connected arc containing $v_0$ for each $i$, and
\item $v_0$ is the terminal point of both $E$ and $E'$.
\end{enumerate}
\end{lem}

To prove Lemma \ref{OrientEdges}, we use Proposition \ref{goodTree} to find a good tree $(T,v_0)$ containing $E$ and so that $v_0$ has valence 1 in $T$. Denote by $e_1$ the third edge at $v_0$ distinct from $E,E'$. Since $E \in E(T)$, we have $E', e_1 \notin E(T)$. 
Lemma \ref{OrientEdges} now follows from:

\begin{lem}\label{OrientEdges2}
Let $G$ be a trivalent graph with no separating edges and let $(T,v_0)$ be a good spanning tree in $G$. Let $e_1 \in E(G-T)$ so that $i(e_1)=v_0$.
Then one can enumerate $E(G-T)-\{e_1\}$ as
$e_2, \dots , e_r$ and orient all of the edges of $G$ so that
\begin{enumerate}
\item there exist positive embedded loops $\al_1, \dots , \al_r \subset G$,
\item for each $i$, we have that $i$ is the smallest index such that $\al_i$ contains $e_i$, and
\item for each $i$, we have that $\al_i \cap ( \cup_{j=1}^{i-1} \al_j )$ is a connected arc containing $e_1$.
\end{enumerate}
\end{lem}

We will need the following definitions in our proof of Lemma \ref{OrientEdges2}.
\begin{df}[$v_-(e)$ and $v_+(e)$]
Let $T$ be a good tree in $G$. Given an edge $e \in E(G)$, one of its endpoints is closer (in $T$) to $v_0$ than the other. We denote by $v_-(e)$ the vertex closer to $v_0$  and by $v_+(e)$ the vertex further from $v_0$. 
\end{df}

\begin{df}
Let $\al$ be an embedded oriented path in the graph $G$ and let $x, y$ be two points in the image of $\al$, i.e $x=\al(s)$ and $y=\al(t)$, for some $s,t$. If $s<t$, then we denote by $[x,y]_\al$ the image of the subpath of $\al$ initiating at $x$ and terminating at $y$, i.e. $\al([s,t])$. \end{df}

\begin{df}[Left-right splitting]
Let $v_0$ be a vertex of a graph $G$ and let $\al$ be an embedded directed loop based at $v_0$. Let $e$ be an edge of $\alpha$, and $m$ the midpoint of $e$. Then the \emph{left-right} splitting of $\al$ at $e$ is
\[ L^e_\al = [v_0, m]_{\al}, \quad   R^{e}_{\al} = [m,v_0]_{\al}. \]
\end{df}

\begin{df}[Aligned edges]
Let $G$ be a graph and $(T,v_0)$ a spanning good tree in $G$. Suppose $\b_1, \b_2 \in E(T)$ satisfy that the vertices 
$\{ v_+(\b_1), v_+(\b_2), v_0 \}$ span a line in $T$. 
Then we call $\b_1, \b_2$ \emph{aligned}. 
If $v_+(\b_1) < v_+(\b_2)$ we say $\b_1$ \emph{lies below} $\b_2$.
\end{df}

\begin{df}[Highlighted subpaths]
Let $G$ be a graph and $(T, v_0)$ a spanning good tree in $G$.
Let $\b_1, \b_2$ be aligned edges in $T$ such that $\b_1$ lies below $\b_2$.
For each $i$, let $\al_{f_i}$ be an embedded loop containing $\b_i$ and $v_0$.
We define $H^{\b_1, \b_2}_{\b_1}$ and $H^{\b_1, \b_2}_{\b_2}$, the \emph{highlighted} subpaths of $\al_{\b_1}$ and $\al_{\b_2}$, respectively, as follows: 
\begin{eqnarray}
\label{eqCase1}H^{\b_1, \b_2}_{\b_1} = L^{\b_1}_{\al_{\b_1}} \text{ and} & H^{\b_1, \b_2}_{\b_2} = R^{\b_2}_{\al_{\b_2}} &  \text{when } i(\b_1) <_T ter(\b_1) \\
\label{eqCase2}H^{\b_1, \b_2}_{\b_1} = R^{\b_1}_{\al_{\b_1}} \text{ and} & H^{\b_1, \b_2}_{\b_2} = L^{\b_2}_{\al_{\b_2}} & \text{when } i(\b_1) >_T ter(\b_1)
\end{eqnarray}
\end{df}

\begin{proof}[Proof of Lemma \ref{OrientEdges2}]
Enumerate $E(G-T)-\{e_1\}$ so that $k<j$ implies that $v_-(e_k) \ngeq v_-(e_j)$.
We prove this lemma by induction on $i$ for $1 \leq i \leq r$. We will define a loop $\al_i$ and orient its previously unoriented edges so that items (1), (2), and (3) of the lemma hold and moreover the following items (4) and (5) hold.
Denote by $\al_e$ the first loop that contains $e$ (for example when $e = e_j \notin E(T)$ then $\al_e = \al_j$). 
\begin{enumerate}\setcounter{enumi}{3}
\item If $\b \in E(T)$ and $\b' \in E(G)$ are such that $v_+(\b) \leq v_+(\b')$ and $\b'$ is oriented, then $\b$ is oriented. 
\item\label{RLconsistant} 
Let $\b_1, \b_2 \in E(T)$ be aligned and $H^{\b_1, \b_2}_{\b_1}, H^{\b_1, \b_2}_{\b_2}$ the corresponding highlighted paths with respect to $\al_{\b_1}, \al_{\b_2}$, then 
$H^{\b_1, \b_2}_{\b_1} \cap H^{\b_1, \b_2}_{\b_2}$
contains no half-edges. 
\end{enumerate}

We include (\ref{RLconsistant}) to ensure that the loop in the induction step is embedded.
\smallskip

\textbf{The base case:} We begin the base of the induction with the edge $e_1 \in E(G-T)$, which  is adjacent with $v_0$.
Define the directed circle $\al_{e_1}= e_1*[t(e_1), v_0]_T$, this is clearly an embedded loop.
We orient the edges of $\al_{e_1}$ accordingly, i.e. $e_1$ is directed away from $v_0$,
and the edges $\b \in [t(e_1), v_0]_T$ are directed toward $v_0$.
Clearly items (1)-(4) hold at this stage, i.e. in the subgraph consisting of precisely $\alpha_1$.

\smallskip

\textbf{The induction hypothesis:} We now assume that we have oriented some subset of $G$ so that:
for each $j \leq i-1$ we have that $e_j$ is oriented. We call $G' = \cup_{j=1}^{i-1} \al_j$ the oriented subgraph. We let $e :=e_i$ and note that $e \notin G'$.

\textbf{The induction step:} Let  $I = [v_-(e),v_+(e)]_T$.
 Let $t_1$ be the edge of $I$ adjacent at $v_-(e)$. 
We claim as follows that $t_1$ is oriented, 
see Figure \ref{partlyOriented}. Indeed, if $v_-(t_1)=v_0$ this follows from the base case. Otherwise, there is an edge $t_3 \in E(T)$ adjacent to $v_-(t_1)$. 
Since $t_3$ is nonseparating, there is an edge $e' \in E(G-T)$ so that $v_-(e') \leq v_-(t_3)$ and $v_+(e') \geq v_+(t_3)$. But since the graph is trivalent, $v_+(e') \geq v_+(t_1)$. 
Now, since $v_-(e') < v_-(e)$, we have that $e'$ is oriented. 
Hence, by item (4) in the induction hypothesis, $t_1$ is oriented. 

\begin{figure}[h]
\begin{center}
\includegraphics[width=4in]{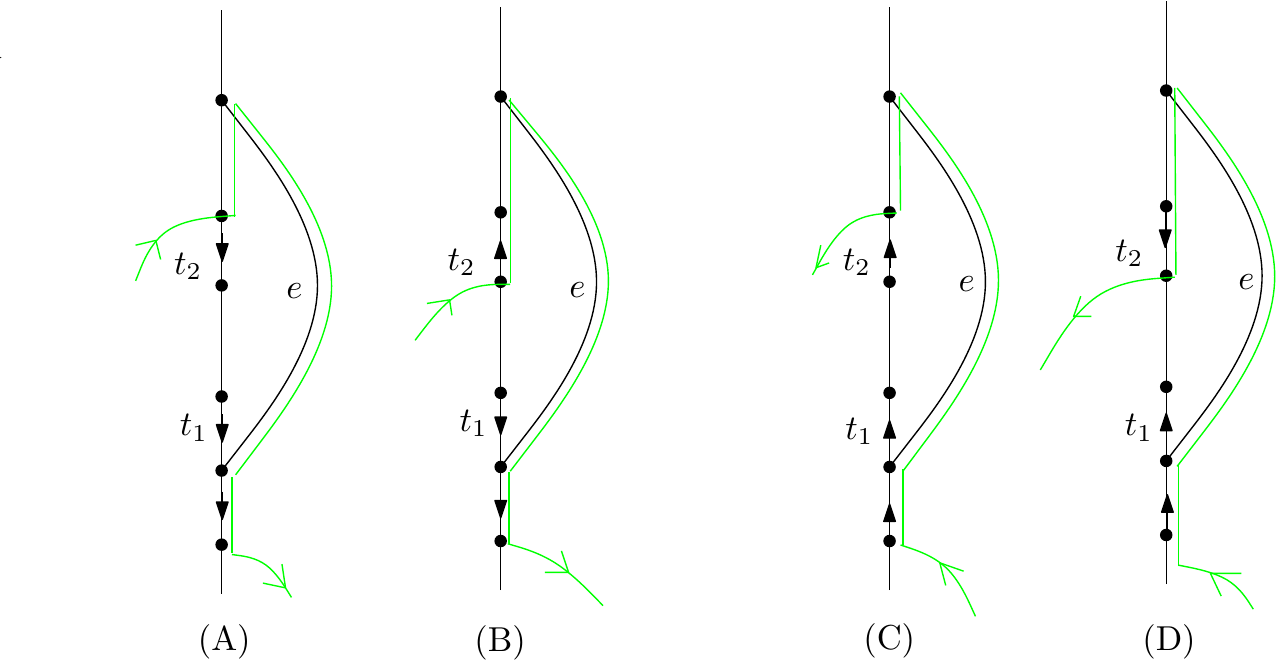}
\caption{The green line indicates the path $\al_{e}$ takes near $e$. The edge $t_1$ is the first edge in $I$ and $t_2$ is the last oriented edge in $I$.\label{partlyOriented}}
\end{center}
\end{figure}

Let  $t_2$ be the last oriented edge of $I$. Denote $J = [v_+(t_2), v^+(e)]_T$. Note that we allow $t_1=t_2$ and leave the adjustments of this case, from cases A and C of Figure \ref{partlyOriented}, to the reader. The loop $\al_e$ is constructed as in Figure \ref{partlyOriented} from the following segments by adding or removing a half-edge of $t_2$ or removing a half-edge of $t_1$ or $t_2$
\begin{equation}
H^{t_1, t_2}_{t_2}, ~ J, ~ e, ~ H^{t_2, t_1}_{t_1}. 
\end{equation} 
The orientation of $\al_e$ is chosen according to the direction of $t_1$. When $i(t_1)>_T t(t_1)$, cases (A) and (B) of Figure \ref{partlyOriented}, we orient $\al_e$ so that $H^{t_1, t_2}_{t_2}$ is the left (first) segment and when $i(t_1)<_T t(t_1)$, cases (C) and (D), we orient so that $H^{t_2, t_1}_{t_1}$ is the first segment. 

We observe that $\al_e$ is indeed embedded as follows. The paths $H^{t_1, t_2}_{t_2}, H^{t_2, t_1}_{t_1} \subset G'$ while $J, e \subset G-G'$, hence they are edge-disjoint. Moreover, by item (\ref{RLconsistant}) for $G'$, the paths $H^{t_1, t_2}_{t_2}, H^{t_2, t_1}_{t_1}$ are half-edge disjoint so that even if we add $t_1$ or $t_2$ they remain edge disjoint. Therefore the path $\al_e$ does not self-intersect in an edge. It cannot self-intersect at a vertex since the graph is trivalent. This proves item (1). 

Note that $\al_e$ is contained in $G' \cup T \cup \{e\}$, hence it does not contain $e_j$ for $j>i$. This proves item (2). Item (3) is also clear. Item (4) is satisfied since for $\b' = e$ and for all $\b \in E(T)$ so that $v_+(\b) \leq v_+(e)$ we have that $\b$ is oriented. 

We are left with proving item (\ref{RLconsistant}) 
for each pair of aligned edges $\b_1, \b_2 \in E(T)$ such that at least one of them is oriented in the $i^{th}$ step, i.e. $\al_\b = \al_i$.
It is less difficult to check that the claim holds when both edges are newly oriented.
We leave this case to the reader and check the case where $\b_1 \subset G' \cap T$ and $f_2 \subset T$ is newly oriented, i.e. $f_2 \subset J$.  
 We illustrate the different cases in Figure \ref{caseA}.

\begin{figure}[h]
\begin{center}
\includegraphics[width=3.5in]{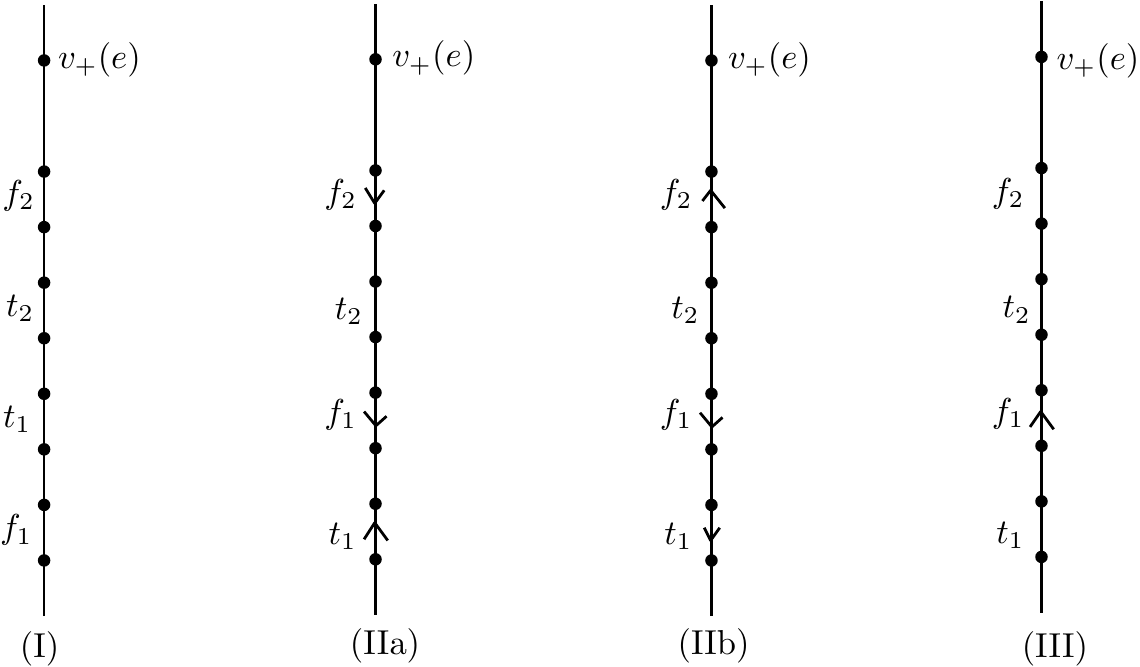}
\caption{Checking Claim \ref{RLconsistant} for $\b_1, \b_2$ depending on their location. \label{caseA} 
In Case I, $\b_1$ lies below $t_1$, and in the other cases $t_1$ lies below $\b_1$. In Cases IIa, IIb, $\b_1$ is pointing down and in Case III, $\b_1$ is pointing up.}
\end{center}
\end{figure}

Suppose $\b_1 \notin I$ (Case I in Figure \ref{caseA}), then for each $\b \in \{\b_2, t_1, t_2\}$, we have $\b_1 <_T \b$. Thus, 
$H_{\b_1}^{\b_1,\b} = H_{\b_1}^{\b_1,\b'} \text{ for } \b,\b' \in \{\b_2, t_1, t_2\}$.

Moreover, by our construction of $\al_{\b_2} = \al_e$, we have $H^{\b_1, \b_2}_{\b_2} \subset H^{\b_1, t_k}_{t_k} \cup J \cup \{t_2\}$ for either $k=1$ or $k=2$. Thus, for the same $k$ we have (see Figure \ref{partlyOriented}) :
\[ 
H_{\b_1}^{\b_1,\b_2} = H_{\b_1}^{\b_1,t_k}\quad \quad  \text{ and } \quad  \quad H^{\b_1, \b_2}_{\b_2} \subset  H^{\b_1, t_k}_{t_k} \cup \{e\} \cup J  \cup \{t_2\}.
\]
Since $J$ is newly oriented, $H_{\b_1}^{\b_1,t_k}$ and $J$ are edge-disjoint. 
By the induction hypothesis, $H_{\b_1}^{\b_1,t_k} \cap  H^{\b_1, t_k}_{t_k}$ contains no edges or half-edges. 
Moreover, if $H^{\b_1, \b_2}_{\b_2} \not\subset H^{\b_1, t_k}_{t_k}\cup J \cup \{e\}$ then $k=2$ and a half of $t_2$ must be in $H^{\b_1, t_2}_{t_2}$. 
Since $H_{\b_1}^{\b_1,t_k} \cap  H^{\b_1, t_k}_{t_k}$ contains no half-edges, then $t_2$ is not contained in $H_{\b_1}^{\b_1,t_k}$. This implies that $H_{\b_1}^{\b_1,\b_2} \cap  H^{\b_1, \b_2}_{\b_2}$.

If $\b_1 \in I$, then there are two classes of cases: $i(\b_1) \geq_T ter(\b_1)$ (Cases IIa, IIb of Figure \ref{caseA}) and $i(\b_1) \leq_T ter(\b_1)$ (Case III).
We will prove Case IIa and leave the others to the reader.
If $\b_1$ is pointing down, then $H^{\b_1, \b_2}_{\b_1} = R^{\b_1}_{\al_{\b_1}}$,  $H^{\b_1, \b_2}_{\b_2} = L^{\b_2}_{\al_{\b_2}}$. 
There are two subcases: $\b_2$ is
pointing down (Case IIa) or up (Case IIb). 
If $\b_2$ is pointing down, then $t_1$ is pointing up (see Figure \ref{partlyOriented}), thus
$L^{\b_2}_{\alpha_e} \subset L^{t_1}_{\alpha_{t_1}} \cup e \cup J$.
In this configuration,
\[ \begin{array}{lll}
H_{\b_1}^{\b_1,t_1} = R^{\b_1}_{\alpha_{\b_1}} & & H_{t_1}^{\b_1,t_1} = L^{t_1}_{\alpha_{t_1}} \\[0.2 cm]
H_{\b_1}^{\b_1,\b_2} = R^{\b_1}_{\alpha_{\b_1}} & & H_{\b_2}^{\b_1,\b_2} = L^{\b_2}_{\alpha_{e}} \subset L^{t_1}_{\alpha_{t_1}}  \cup \{e\} \cup J
\end{array}\]
Hence, the fact that $ H_{\b_1}^{\b_1, t_1} \cap H_{t_1}^{\b_1, t_1}$ contains no half-edges implies that $ H_{\b_1}^{\b_1, \b_2} \cap H_{\b_2}^{\b_1, \b_2}$ contains no half-edges.
The other cases are similar.
\end{proof}

\subsection{Rose-to-graph fold line}\label{roseToGraphSec}
Given a point $x$ whose underlying graph is trivalent with no separating edges, we wish to find a rose-point $x_0$ and a line in $\os$ from $x_0$ to $x$. This is done by simultaneous folding as defined below. 

\begin{df}[Rose-to-graph fold line $\calF(x_0, \{s_{ij}\})$]
Let $x_0 = (R, \mu, \ell_0)$ be a point in $\uos$ whose underlying graph is a rose with $r$ petals. There are $K = r(r-1)$ turns  and we enumerate them in any way $\{\tau_i\}_{i=1}^K$. Let $\overrightarrow{s} \in \RR^{K}$ be a nonnegative vector so that  $s_i$ is no greater than the length of each edge in the turn $\tau_i$. 
Given the data $(x_0, \ov{s})$ we construct a continuous family of graphs $\{x_t\}$ for $0 \leq t \leq T = \sum s_i$, and maps $f_{t,0} \from x_0 \to x_t$ as follows. In the $i^{th}$ step let $\tau_i = \{e_j, e_m\}$ and fold initial segments of length $s_i$ in $f_{t,0}(e_j)$ and $f_{t,0}(e_m)$. 
We caution that these $f_{t,0}$ are not always homotopy equivalences. However, if $f_{T,0}$ is a homotopy equivalence, then for each $t<T$ the map $f_{t,0}$ is a homotopy equivalence. In this case we get a path 
\[ \begin{array}{rl}
\calF(x_0,\ov{s}) \from [0,T] &\to \uos \\ [0.2 cm]
\calF(x_0, \ov{s})(t) &= x_t.
\end{array} \] 
We denote its projectivization by $\ol{\calF} = q(\calF)$.
\end{df}

\begin{lem}\label{RtoRfoldlineLemma}
Let $G$ be a trivalent graph such that $\pi_1(G) \cong F_r$. Then for each $x \in \uos$ with underlying graph $G$, there exists a point $x_0$ whose underlying graph is a rose and there exists a nonnegative vector  $\ov{s}$ in $\RR^{K}$, where $K=r(r-1)$, so that
\[ x = \calF(x_0, \ov{s})(T). \]
Additionally, $x_0, \ov{s}$ are linear functions of the lengths of $E(G)$ as they vary throughout the unprojectivized simplex. (As above, $T = \sum s_{i}$.)
\end{lem}

\begin{proof}
Let $E_1, \dots E_{3r-3}$ be the edges in $G$. Lemma \ref{OrientEdges} provides a decomposition of $G$ as
$G = \cup_{i=1}^r \al_i$.
Let $R = \sqcup_{i=1}^r \al_i/\{v_0\}$, then $R$ is an $r$-petaled rose. Let $e_1, \dots e_r$
be the edges of $R$. Let $\ell_R(e_i) = \ell_G(\al_i)$. We write $\al_i$ as a sequence of edges
$\al_i = E_{m(i,1)} \cdots E_{m(i,k_i)}$.
Then
\begin{equation}\label{eqLengths}
\ell_R(e_i) = \ell_G(\al_i) = \sum_{j=1}^{k_i} \ell_G(E_{m(i,j)}).
\end{equation}
Let $x_0$ denote $R$ with these edge lengths. 
There is a natural map $\theta \from x_0 \to x$ defined by the inclusion of $\al_i$ in $G$.
This is a quotient map and, moreover, $\theta|_{e_i}$ is an isometry.

Recall that the intersection of $\al_i$ and $\al_j$ is an arc containing $v_0$. Let $a,b$ be the endpoints of the arc. For $\tau_k = (e_i, e_j)$ and $\tau_m = (\overline{e_i}, \overline{e_j})$ we define 
\begin{equation}\label{eqLengths2}
s_{k} = l_G([v_0,a]_{\al_i}) \quad \textup{and} \quad s_{m} = l_G([b,v_0]_{\al_i}) . 
\end{equation}
Consider the folding line $\calF(x_0, \ov{s})$. By the definitions, $\theta$ is precisely the map $f_{T,0} \from x_0 \to x_T$, since the points identified by $\theta$ are precisely those that are identified in the folds. Therefore, $x_T$ equals the point $x$ that we started with.

Moreover, $\theta$ is a homotopy equivalence by Lemma \ref{OrientEdges2}(2). Therefore, $\calF(x_0, \ov{s})$ is a path in unprojectivized Outer Space.  Equations \ref{eqLengths}, \ref{eqLengths2} show that the dependence of $l_i$ and $s_{k}$ on edge lengths in $x$ is linear.
\end{proof}

\begin{lem}\label{ContOfRtRfoldingLingLemma}
Suppose that $x = \calF(x_0, \ov{s})(T)$. Then there exists a neighborhood of $(x_0, \ov{s})$ so that for each $(y_0, \ov{u})$ in this neighborhood, the endpoint $y:=\calF(y_0, \ov{u})(T')$ of the fold line $\calF(y_0, \ov{u})$ lies in the same unprojectivized open simplex in $\uos$ as $x$.

Additionally, the edge lengths of $y$ are linear combinations of the edge lengths of $y_0$ and $\ov{u}$.
\end{lem}

\begin{proof}
Consider the positive edges $E_1, \dots E_m$ in $G$ and let $G = \cup_{i=1}^r \al_i$ be the decomposition guaranteed by Lemma \ref{OrientEdges}. The edge $E_i$ is contained in a loop $\al_{j(i)}$. Since 
$G$ is a trivalent graph, at each endpoint $\{v,w\}$ of $E_i$ there is an edge, $E_k, E_d$ resp., not 
contained in $\al_i$. The edges $E_k, E_d$ are contained in $\al_{j(k)},\al_{j(d)}$ resp. 
Now $v_0, v \in \al_{j(i)} \cap \al_{j(k)}$. Thus, by Lemma \ref{OrientEdges}(3), either
$[v_0, v]_{\al_{j(i)} } \subseteq \al_{j(i)}  \cap \al_{j(k)}$ or
$[v, v_0]_{\al_{j(i)} } \subseteq \al_{j(i)}  \cap \al_{j(k)}$. This situation is similar for $v_0,w$. Therefore,
\begin{equation}\label{eqOppLengths}
\ell_G(E_i) = 
\begin{cases}
s_{m}- s_{n}  & \text{when }\tau_m = (e_{j(d)}, e_{j(i)}) \text{ and } \tau_n=(e_{j(k)}, e_{j(i)}) \\[0.1 cm]
s_n - s_m & \text{when } \tau_m = (\overline{e_{j(d)}}, \overline{e_{j(i)}}) \text{ and } \tau_n=(\overline{e_{j(k)}}, \overline{e_{j(i)}}) \\[0.1 cm]
 |\ell(\al_{j(i)})-(s_n +s_m)|  & \text{otherwise }\end{cases}
\end{equation}
Note this dependence of the $\ell_G(E_i)$ on the variables $s_{m}$ will be the same for all points in the same unprojectivized simplex since they only depend on the loop decomposition. We also get that the dependence of $\ell_G(E_i)$ on $s_{ij}$ and $\ell(e_i)$ is linear.
Let $U$ be the open subset of the unprojectivized rose simplex cross $\RR_+^{r(r-1)}$ so that each expression in the right-hand side of (\ref{eqOppLengths}) is positive for each $i$. 
This is an open set containing $(x_0, \ov{s})$. 
For any point $y$ in the unprojectivized simplex of $x$ one can use (\ref{eqLengths}) and (\ref{eqLengths2}) 
to get $y_0$ and $\ov{u}$ so that 
$y = \calF(y_0, \ov{u})$. Thus $(y_0, \ov{u})$ solves the equations in (\ref{eqOppLengths}) for the edge lengths of $y$, hence $(y_0, \ov{u}) \in U$. 
Thus, there exists a solution of (\ref{eqOppLengths}) in $U$ for each choice of edge lengths of $G$. We see that the dependence of the edge lengths $l_G(E_i)$ on $\ov{s}$ and $\ell(e_i)$ is linear. 
\end{proof}

\subsection{Folding a transitive graph to a rose}

\begin{df}
A \emph{transitive} graph $G$ is a directed graph $G$ with the following property: for any two vertices $w,w'$ there exists a directed path from $w$ to $w'$.

Note that it is enough to check that for any choice of preferred vertex $v$, there exists a directed path to and from each other vertex $v'$.
\end{df}

The proof of the following is left to the reader.
 
\begin{obs}\label{ObsContInd}
Let $G$ be a directed graph and let $f \from G \to G'$ be a direction matching fold of two oriented edges $e_1,e_2$ in $G$, starting at a common vertex $v$. Then:
\begin{enumerate}
\item If $G$ is transitive, then $G'$ is transitive.
\item If the lengths of the edges of $G$ are rationally independent, then the lengths of the edges of $G'$ are rationally independent.
\end{enumerate}
\end{obs}

\begin{lemma}\label{graphToRoseFolding}
Let $G$ be any transitive graph with rationally independent edge lengths and let $\{E, E'\}$ be either a positive or negative turn. Then there exists a fold sequence $f_1, \dots, f_k$ containing only direction matching folds and satisfying that $f_k \circ \dots \circ f_1(G)$ is a rose. Further, assuming $G$ is trivalent, we may choose $f_1$ so that it folds the turn $\{E,E'\}$. 
\end{lemma}
\begin{proof}
We perform the following steps:

\textbf{Step 1:}
Let $c(G)$ be the number of directed \emph{embedded} paths in $G$ between (distinct) vertices. For 
each pair of vertices $w,w'$ there exists a directed path from $w$ to $w'$, thus it follows that there 
exists an embedded directed path from $w$ to $w'$. We decrease $c(G)$ by folding two directed embedded paths $\al,\beta$ so that $i(\al) = i(\beta)$ and $ter(\al) = ter(\beta)$ and
$\al \cap \beta = \{i(\al), ter(\al) \}.$
Note that the edges in $\al$ and $\beta$ are distinct and thus we are consecutively performing regular Stallings folds.
Also, if $G$ is trivalent and we choose a decomposition as in Lemma \ref{OrientEdges}, then we can choose the first $\al$ to contain $E$ and the first $\beta$ to contain $E'$ and fold $\al, \beta$ so that the first combinatorial fold folds the turn $\{E,E'\}$.
Denote the new graph by $G'$. Then, by Observation \ref{ObsContInd}, $G'$ is transitive and has rationally independent edge lengths.
The complexity has decreased, i.e. $c(G')<c(G)$.

At the end of this step we may assume that we have a connected graph $G$ so that
$G = \coprod_{i=1}^n \gamma_i / \sim$,
where $\gamma_i$ is a circle and for each $i \neq j$:
$\gamma_i \cap \gamma_j$ is either empty or a single point (see Figure \ref{circleWheel}). We call such a graph a \emph{gear graph}. Notice that given a gear graph, such a decomposition into circles is unique up to reindexing
\begin{figure}[ht]
\begin{center}
\includegraphics[width=1in]{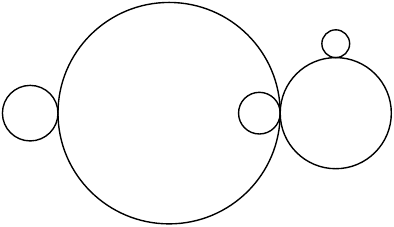}
\caption{This\label{circleWheel} graph is an example of a gear graph.}
\end{center}
\end{figure}

\textbf{Step 2:}
For a gear graph, we define a new complexity. Let $V$ be the set of vertices (of valence $>2$) and let $m=\max\{ val(w) \mid v_0 \neq w \in V \}$. Define the complexity:
\[ c(G) = (|V|, m). \]

Let $w$ be a vertex that realizes the valence $m$ and let $\gamma_1, \gamma_2$ be circles so that $w=\gamma_1\cap \gamma_2$ and suppose that $v \neq w$ is a vertex on $\gamma_2$ (see Figure \ref{circleWheel}).
Let $\gamma_2 = \al \beta$ where $i(\al) = v = ter(\beta)$ and $ter(\al) = w = i(\beta)$. By folding, wrap $\beta$ over $\gamma_1$ until $v$ is on the image of $\gamma_1$ (we may have to wrap $\beta$ multiple times over $\gamma_1$). Now there are two paths from $v$ to $w$: $\al$ and $\gamma_1'$, the remaining part of $\gamma_1$. Fold $\overline{\al}$ over $\overline{\gamma_1'}$.
$G'$ is a gear graph. Moreover, the valence of the image of $w$ decreases ($w$ may now even have valence 2).
Thus $c(G')< c(G)$.
When all vertices other than $v_0$ have valence $2$  we can no longer decrease the complexity, and $G$ will be a rose.

\end{proof}

\begin{figure}[H]
\begin{center}
\includegraphics[width=4in]{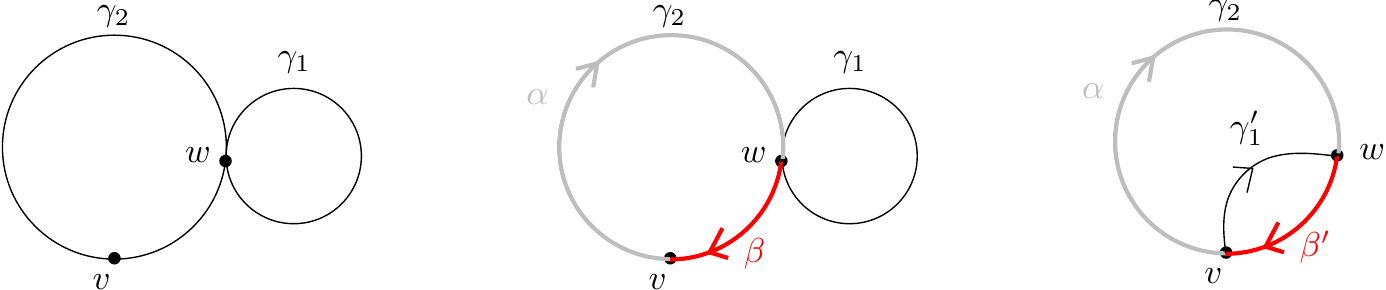}
\caption{Step 2 ($\beta'$ is the remaining portion of $\beta$, after it is wrapped around $\gamma_1$) \label{step2}}
\end{center}
\end{figure}

\section{Existence and continuity of rose-to-rose fold lines}\label{S:ContinuityCloseness}
In Subsection \ref{roseToGraphSec} we defined a rose-to-graph fold line $[x_0, x] = \calF(x_0, \ov{s})$, given a point $x$ and a loop decomposition of its underlying graph. We also had two continuity statements: (1) $x_0,\ov{s}$ vary continuously as a function of $x$ by Lemma \ref{RtoRfoldlineLemma} and (2) $\calF(x_0, \ov{s})$ vary continuously as a function of $x_0,\ov{s}$ by Lemma \ref{ContOfRtRfoldingLingLemma}.
We need similar existence and continuity statements for the full rose-to-rose fold line.

\begin{prop}\label{existenceRtRFoldLines}
Let $x$ be any point of $\uos$ satisfying that $q(x)$ is in a top-dimensional simplex of reduced Outer Space with rationally independent edge lengths. Let $\{E,E'\}$ be a pair of adjacent edges in the underlying graph of $x$. Then there exists a positive rose-to-rose fold line, which we denote by $\calR(x_0,\ov{s}) \from [0,L] \to \os$, containing $x$ and containing the fold of $\{E,E'\}$. 
\end{prop}

\begin{proof}
By Lemma \ref{RtoRfoldlineLemma} there exists a rose point $x_0$ and a vector $\ov{s}$ so that 
the rose-to-graph fold segment $\calF(x_0, \ov{s})$ terminates at the point $x$. We consider the orientation on $x$ given by the loop decomposition in Lemma \ref{OrientEdges}. We may apply Lemma 
\ref{graphToRoseFolding} to obtain a fold sequence $f_1, \dots, f_k$ which terminates in a rose $z'$ with some valence-2 vertices and so that $f_1$ folds the turn $\{E,E'\}$. Removing the valence-2 vertices gives a rose, which we denote by $z$. 
The line just described will be denoted $\calR(x_0, \ov{s})$. It satisfies the statement in the theorem. 
\end{proof}

\noindent \textbf{Notation:} Let $x_0$ be a rose, let $\ov{s} \in \RR^{r(r-1)}$, and let $\calR(x_0,\ov{s})$ be the fold line defined by these parameters. (We call such a fold line a \emph{rose-to-rose} fold line.)
We will denote by $x$ the trivalent metric graph at the end of the rose-to-graph fold segment, i.e. $x = \calR(x_0, \ov{s})(\sum s_i)$,  
by $z$ the end-point of the rose-to-rose fold segment, and by $z'$ the point in the graph-to-rose fold 
segment from which $z$ is obtained by removing the valence-2 vertices. 

\begin{df}[Proper fold line]
Let $\calR(x_0, \ov{s})$ be a rose-to-rose fold line and let $f_1, \dots, f_k$ be the sequence of combinatorial folds from $x$ to $z'$. If for each $l$ the fold $f_l$ is a proper full fold, then we will say that $\calR(x_0, \ov{s})$ is a \emph{proper fold line}. 
\end{df}

For example, for each $x$ as in Proposition \ref{existenceRtRFoldLines}, the fold line constructed in the proposition is a proper fold line since the edge lengths in $x$ are rationally independent.

\begin{prop}\label{contOfRtRFoldLines}
For each proper rose-to-rose fold line $\calR(x_0,\ov{s})$ and $\eps>0$, there exists a neighborhood $U$ of $(x_0, \ov{s})$ so that for any point $(y_0, \ov{u}) \in U$:
\begin{enumerate}

\item The endpoints of the rose-to-graph fold segments
\[ y: = \calF(y_0,\ov{u})(T'), \quad x: = \calF(x_0, \ov{s})(T) \]
lie in the same unprojectivized open simplex and are $\eps$-close.


\item The sequence of combinatorial folds from $x$ to $z'$ 
appearing in the graph-to-rose fold segment 
is allowable in $y$.

\item 
Let $\calR(y_0, \ov{u})$ be the fold line defined by concatenating $\calF(y_0, \ov{u})$ with the fold segments from (2), then 
the terminal points $w: = \calR(y_0, \ov{u})(L')$ and $z:=\calR(x_0, \ov{s})(L)$ are $\eps$-close.
\end{enumerate}
\end{prop}

\begin{proof}
By Lemma \ref{ContOfRtRfoldingLingLemma} there exists a neighborhood $U$ of $(x_0, \ov{s})$ so that
for each $(y_0, \ov{u}) \in U$ the fold line $\calF(y_0,\ov{u})$ terminates at some point $y$ lying in the same unprojectivized top-dimensional simplex as $x$.
Since the edge lengths of $y$ vary continuously with the edge lengths of $y_0$ and $\ov{s}$, we can 
make $U$ smaller, if necessary, to ensure that $y$ and $x$ are $\eps$-close. This proves (1). 

Let $f_1, \dots , f_k$ be the combinatorial fold sequence from $x$ to $z'$, as described in 
Lemma \ref{graphToRoseFolding}.
For each combinatorial proper full fold folding $e_i$ over $e_j$, there corresponds a square $m \times m$ folding matrix (here $m>r$), which we denote $T'_{ij} = ( a_{kl})$, so that $a_{kl} = 1$ for $k=l$ and $a_{ij} = -1$, but otherwise $a_{kl} = 0$ when $k \neq l$ (compare with Definition \ref{dfFoldingMatrices}). Let $T_1', \dots , T_k'$ be the fold matrices corresponding to $f_1, \dots, f_k$. 
Then just as in Lemma \ref{L:SequenceAllowable}, the combinatorial fold sequence $f_1, \dots, f_k$ is allowable in the point $y$ if and only if for each $l$ the vector $T_l' \cdots T_1'(\ell(y))$ is positive for each $1 \leq l \leq k$. This defines an open neighborhood of $x$ where item (2) holds.  
%
Item (3) follows from the fact that matrix multiplication is continuous. 
\end{proof}

\begin{lemma}\label{changeOfMetricMatrixPos}
Let $h \from x_0 \to z$ be a homotopy equivalence representing the map from the initial rose $x_0$ to the terminal rose $z$ in a rose-to-rose fold line.
Let $H$ be a matrix representing the change-of-metric from the rose $z$ to the rose $x_0$. Then
$H$ is a nonnegative invertible integer matrix and it equals the transition matrix of $h$.
\end{lemma}

\begin{proof}
Let $x_0$ be the initial rose with edges $e_1, \dots, e_r$ and let $z$ be the terminal rose with edges $e_1', \dots , e_r '$.
Let $z'$ be the graph on the fold line just before $z$, i.e. $z$ is obtained from $z'$ by unsubdividing at all of the valence-2 vertices.
Let $h\from x_0 \to z$ and $g \from x_0 \to z'$ be the relevant homotopy equivalences. 
Since $g$ is a subdivision followed by folding maps, each of which is a positive map, we have for each $i$ that $g|_{e_i}$ is a local isometry. 
Thus, $h|_{e_i}$ is a local isometry. 
Moreover, $h$ maps the unique vertex of $x_0$ to the unique vertex of $z$. Suppose $h(e_i)$ contains a 
part of an edge $e_i'$, then $h(e_i)$ contains a full appearance of $e_i'$ (since there is no backtracking and the vertex maps to the vertex). Therefore, $h|_{e_i}$ is   an edge-path in $z$. Thus, we may write
\begin{equation}\label{TerToIni} l(e_i) = \sum_{j}m(i,j) l(e_j'), \end{equation} where the $m(i,j)$ are the nonnegative integer entries of the transition matrix of $h$.
By Equation \ref{TerToIni}, the change-of-metric matrix from $z$ to $x_0$ coincides with the transition matrix of the homotopy equivalence $h$. 
Therefore, $H$ is nonnegative, and integer.
Moreover, since all the folds in a rose-to-rose fold line are direction matching, $h$ is a positive map. Thus, 
$H$ is equal to $Ab(h)$, the map induced by $h$ (viewed as an automorphism) by abelianization. Therefore, $H$ is invertible.
\end{proof}

\begin{lemma}\label{AllowableNbhd}
For each $\eps>0$ and proper fold line $\calR(x_0, \ov{s})$, there exists a neighborhood $U$ of the terminal rose $z$ such that, for each $w \in U$, there exists a proper rose-to-rose fold line $\calR(y_0, \ov{u})$ terminating at $w$ satisfying that 
\begin{enumerate}
\item the top graphs $x,y$ are $\eps$-close,
\item the combinatorial fold sequence corresponding to the graph-to-rose segments are the same in both lines, and
\item the change-of-metric matrix for both fold lines is the same.
\end{enumerate}
\end{lemma}

\begin{proof}
We prove (1) and (2). By Lemma \ref{changeOfMetricMatrixPos}, the change-of-metric matrix $H$ from $z$ to $x_0$ is nonnegative. Thus, for any $w$ in the same unprojectivized simplex as $z$, we have that $H \ell(w)$ is also positive. 
By Proposition 
\ref{contOfRtRFoldLines}, there exists a neighborhood $V$ of $x_0$ so that for each $y_0 \in V$  there exists a vector $\ov{u}$ so that if $y$ is the top graph of the fold line $\calF(y_0, \ov{u})$, then $x$ and $y$ are $\eps$-close and their combinatorial fold sequences are the same. The neighborhood $U$ can be taken in $H^{-1}(V)$. 
This proves (1) and (2). 
Since the combinatorial folds are the same, the transition matrix $y_0 \to w$ is the same as the 
transition matrix $x_0 \to z$. 
(3) follows from Lemma \ref{changeOfMetricMatrixPos}. 
\end{proof}

\begin{df}[Rational rose-to-rose fold lines]
A rose-to-rose fold line $\calR(x_0, \ov{s})$ is called \emph{rational} if for each edge $e$ in $G$ and for each loop $\al_i$ in the loop decomposition of the underlying graph of $x$, the quotient $\frac{l(e,x)}{l(\al_i, x)}$ is rational.
\end{df}

\begin{prop}\label{ApproxByRRtRFoldLines}
For each $\eps>0$ and for each $x$ in an unprojectivized top-dimensional simplex of reduced Outer Space, with rationally independent edge-lengths, there exists a rational proper rose-to-rose fold line passing through some $x'$ in the same open unprojectivized simplex as $x$, satisfying that $d(x',x)< \eps$.  
\end{prop}

\begin{proof}
Let $x$ be a point in an unprojectivized top-dimensional simplex and having rationally independent edges. 
Let $\calR(x_0,\ov{s})$ be a proper rose-to-rose fold line containing $x$. 
Let $V$ be a neighborhood of $(x_0, \ov{s})$ guaranteed by Proposition \ref{contOfRtRFoldLines}. Let $U$ be a neighborhood of $x$ in the same top-dimensional simplex and such that for each $y \in U$ there exists some $(y_0, \ov{u}) \in V$ with $y$ the terminal point of $\calF(y_0, \ov{u})$. 
This is possible by Lemma \ref{RtoRfoldlineLemma}. 
Let $x'$ be a point in $U$ which is $\eps$-close to $x$ and such that the ratios $\frac{l(e',x')}{l(\al_i, x')}$ are rational. Then the resulting fold line through $x'$ will have the required properties. 
\end{proof}

\section{Constructing the fold ray}\label{S:RayConstruction}

Enumerate $\mP_r$ by $\{ v_i \}_{i=1}^\infty$ (see \sref{s:BrunsAlgorithm}).
For each $i$ there are countably many rational proper rose-to-rose fold lines. Thus there are countably many rational proper rose-to-rose fold lines terminating at a rose with length-vector $v_i$. For each such fold line $\calR_{ij}$, let $F_{ij}$ denote its fold matrix and $H_{ij}$ its inverse - an invertible nonnegative matrix.
Let $U_{ij}$ denote the neighborhood of $v_i$ from Lemma \ref{AllowableNbhd} (for $\eps = 1$ will suffice), i.e. for each $w \in U_{ij}$ there exists a proper fold line terminating at $w$, passing through the same simplices as $\calR_{ij}$ and satisfying that their fold matrices are the same. 

Since each $A_i=A_{v_i}$ is a positive matrix, for each $i,j$ there exists an integer $n(i,j)$ satisfying that for each $n>n(i,j)$ we have $A_i^{n}(\RR^r_+) \subset U_{ij}$.

Recall that $g_{v_i}$ has a decomposition into fold automorphisms obtained from Brun's algorithm, this induces a decomposition of $A_i$ to unfolding matrices and $A_i^{-1}$ into folding matrices.
We then create a sequence of pairs denoted $\{a_k\}$ which satisfies: 
\begin{enumerate} 
\item If $a_k=(i,j)$ for an odd $k$ then there exists some $n$ such that $n>n(i,j)$ so that $a_{k+1}=(i,n)$. 
\item For each $i, N \in \NN$ there exists an $n>N$ and an even $k$ so that $(i,n)=a_k$,
\item For each $i, j \in \NN$, there exist infinitely many odd $k$'s so that $a_k=(i,j)$, 
\end{enumerate}

To each $a_k$ in this sequence we attach an (unfolding) matrix or sequence of matrices and  automorphisms: if $k$ is odd and $a_k=(i,j)$ we attach the matrix $H_{ij}$ related to the rose-to-rose fold line $\calR_{ij}$ and define $f_k$ to be its change-of-marking automorphism. And if $k$ is even and $a_k=(i,n)$ we attach a sequence of unfolding matrices according to a Brun's algorithm decomposition $A_i^n = (M_1^i \dots M_k^i)^n$ and a sequence of fold automromphisms. We get a sequence of matrices, which we denote by $\{D_l\}_{l=1}^\infty$. For each $l$, either $D_l = H_{ij}$ or $D_l = M^i_j$ for some $i,j$. We emphasize that we have not decomposed $H_{ij}$. Moreoverm we get a sequence of automorphisms $\{f_l\}_{l=1}^\infty$. Let $\{Z_l\}_{l=1}^\infty$ denote the sequence of folding matrices, i.e. $Z_l = D_l^{-1}$ for each $l$. 

\begin{df}[Ray $\mathfrak{R}$]\label{construction}
We construct as follows the geodesic fold ray $\mathfrak{R}$ that we later prove is dense. 
Let $x_0$ be a rose in the unprojectivized base simplex (i.e. having the identity marking) with the length vector $w_0$ provided by Lemma \ref{l:DefiningMetric} for the matrix sequence $\{D_l\}$ that we defined in the paragraph above. 
For each $k$ let $x_k$ be the rose with length vector inductively defined as $w_k=Z_kw_{k-1}$ and with the marking $f_k \circ \dots \circ f_1$. 
For each $l$, if $Z_l$ is a single fold matrix coming from the matrix decomposition of $A^n_i$, then let the line from $x_l$ to $x_{l+1}$ be the single proper full fold fold line corresponding to the matrix $Z_l$. This fold is allowable since $Z_l w_{l} = w_{l+1}$ is positive. 
If $Z_l=H_{i,j}$ for some $i,j$ coming from $a_k = (i,j)$, then $a_{k+1}=(i,n)$ for $n>n(i,j)$. Hence there is a number $s$ so that the following $s$ matrices $\{D_d\}_{d=l+1}^{l+s}$ are the matrices of the decomposition of $A_i^{n}$. Therefore, $w_{l+1} = A^n_i(w_{l+s})$ for $n>n(i,j)$. Thus $x_{k+1} \in U_{ij}$, so the  fold line corresponding to $\calR_{ij}$ is allowable in $x_{k}= H_{ij}(x_{k+1})$.
We insert this fold line between such $x_k$ and $x_{k+1}$.
This defines a fold ray connecting the $x_k$'s.
\end{df}

\begin{mainthmB}\label{main2}
For each $r \geq 2$, there exists a geodesic ray $\tilde \gamma \from [0, \infty) \to \ros$ so  that the projection of $\tilde \gamma$ to $\uros/\out$ is dense. 
\end{mainthmB}

\begin{proof}
We will show that, for each $r \geq 2$, the fold ray $\mathfrak{R}$ of Definition \ref{construction} is contained in $\uros$ and projects densely into $\mathcal{U}_r$. The ray is geodesic by Corollary \ref{PositiveFoldLineGeodesics}.

To prove that $\mathfrak{R}$ never leaves $\ros$, i.e. contains no graph with a separating edge, it will suffice to show that at each point $x \in \mathfrak{R}$, the underlying graph can be directed so that it is a transitive graph. This clearly holds for each proper full fold of a rose, hence for the fold sequences coming from the decompositions of the $g_{v_i}^k$. Moreover, for each $i,j$  $\calR_{ij}$ consists of transitive graphs since all folds are direction matching, see Observation \ref{ObsContInd}(2).

Let $x$ be a point in an unprojectivized top-dimensional simplex with rationally independent edge lengths. 
Let $G$ be its underlying graph and $\{E, E'\}$ a turn. 
By Theorem \ref{existenceRtRFoldLines} we can construct a proper positive rose-to-rose fold line $\calR = \calR(x_0, \ov{s})$ containing $x$ and the combinatorial fold $f_1$ of $\{E, E'\}$ directly after $x$. 
We may assume that the terminal point of this rose-to-rose fold line $z$ lies in $\sig_0$. 
For each $\eps>0$, by Lemma \ref{ApproxByRRtRFoldLines}, there exists a proper \emph{rational} rose-to-rose fold line $\calR'$ containing a point $x'$ in the same unprojectivized open simplex as $x$, so that $x,x'$ are $\eps$-close, and so that the fold $f_1$ is the fold following $x'$.
Let $H$ be the unfolding matrix corresponding to $\calR'$.

For each $\eps>0$, there exists, as in Lemma \ref{changeOfMetricMatrixPos}, an open neighborhood $U$ of the terminal point of $\calR'$ so that for each $w \in U$, there exists a proper rose-to-rose fold line $\mathcal{R}'(y_0, \ov{u})$, terminating at $w$, such that the top graphs $x',y$ are $\eps$-close, the combinatorial fold sequence in the graph-to-rose segments are the same, and the change-of-metric matrix for $\mathcal{R}'(y_0, \ov{u})$ is $H$.
Since the set of PF eigenvectors is dense, there exists an $i$ so that the PF eigenvector $v_i \in U$. 
Hence, there exists a rose-to-rose fold line $R_{ij}$ passing through the same unprojectivized simplices as $\calR'$ 
and having the same change-of-metric matrix $H_{ij} = H$. 
Moreover, there exists a point $x_{ij}'' \in R_{ij}$ in the same unprojectivized top-dimensional simplex as $x'$ and $\eps$-close to $x'$. 
By Definition \ref{construction}, there exist infinitely many $k$'s so that the fold line between $x_k$ and $x_{k+1}$ is the one passing through the same unprojectivized simplices as $R_{ij}$ (hence $\calR'$). In fact, these occur before arbitrarily high powers of $g_{v_i}$, so that they terminate arbitrarily close to a rose with length vector $v_i$. Let $k$ be such a number and let $\Psi_k$ be the composition of the automorphisms $f_1, f_2, \dots $ up to $x_{k}$. 
Thus, by Lemma \ref{changeOfMetricMatrixPos}, there exists a point $\xi \in [x_{k-1}, x_{k}] \cdot \Psi_k^{-1}$ in the same unprojectivized top-dimensional simplex as $x_{ij}''$ and $\eps$-close to $x_{ij}''$. Hence, $\xi$ is the point on the ray defined in Definition \ref{construction} that is $3\eps$-close to our original point $x$ and the fold immediately after $\xi$ is the one folding the turn $\{E,E'\}$. 
\end{proof}

\begin{mainthmA}\label{main1}
For each $r \geq 2$, there exists a geodesic fold ray in the reduced Outer Space $\ros$ whose projection to $\ros/\out$ is dense. 
\end{mainthmA}

\begin{proof} This is an immediate corollary of Theorem B.
\end{proof}

\section{Appendix: Limits of fold geodesics.}{\label{AppendixBoundary}}

In many cases, as in the case of the geodesic that we construct in this paper, a concatenation of fold segments $\{\gamma_i \colon [i,i+1] \to \uos\}_{i=1}^\infty$ that glue together to a ray $\gamma \colon [0,\infty) \to \uos$, projecting under $q$ to a Lipschitz geodesic, satisfies the properties of a semi-flow line below.

\begin{definition}[Semi-flow line] (\cite[pg. 3]{hm11}, definition of a ``fold line'')
A semi-flow line in unprojectivized Outer Space is a continuous, injective, proper function 
$\RR \to \uos$ 
defined by a continuous 1-parameter family of marked graphs $t \to G_t$ 
for which
there exists a family of homotopy equivalences $h_{ts} : G_s \to G_t$ 
defined for $s \leq  t \in \RR$,
each of which preserves marking, such that the following hold:
\begin{enumerate}
\item
Train track property: For all $s \leq t \in \RR$, the restriction of $h_{ts}$ to the
interior of each edge of $G_s$ is locally an isometric embedding.
\item 
Semiflow property: $h_{ut} \circ h_{ts} = h_{us}$ for all $s \leq t \leq u \in \RR$ and $h_{ss} \colon G_s \to G_s$ is the identity for all $s \in \RR$.
\end{enumerate}
\end{definition}

Handel and Mosher (\cite[\S 7.3]{hm11}) prove each flow line converges (in the axes or Gromov-Haussdorff topologies) to a point $T_\infty$ in $\overline{\calX_r}$, the direct limit of the system.  

\begin{theorem}
For any flow line, its direct limit $T_\infty$ is an $F_r$-tree that has trivial arc stabilizers. Hence, in particular, not every point in the boundary is the direct limit of a flow line.
\end{theorem}

\begin{proof}
We lift the maps to a direct system on trees $\{f_{ts}: T_s \to T_t\}$ that are $F_r$-equivariant, restrict to isometric embeddings on edges, and form a direct system. In \cite[\S 7.3]{hm11} it is shown that the maps $f_{\infty s} \colon T_s \to T_\infty$ that are given by the direct limit construction are also edge-isometries and $F_r$-equivariant. 

Assume $x,y \in T_\infty$ are such that $\gamma \in Stab[x,y]$. Without generality loss, assume $d(x,y)=1$ in $T_\infty$. We will show this leads to a contradiction.
For each $t \in [a, \infty)$, we denote by $A_t(\gamma)$ the axis of $\gamma$ in $T_t$.
Letting $\eps = \frac{1}{5}$, there exists some $s \geq a$ so that:
$f_{\infty s}(x_s) = x$, $f_{\infty s}(y_s) = y$, and 
$d(x_s, \gamma x_s)< \eps , \quad d(y_s, \gamma y_s) <\eps$, and  $|d(x_s, y_s) - 1|<\eps$.

Since $f_{ts}$ is distance non-increasing for all $t,s$, we have for all $t \geq s$ that
\[ 1 \leq d(f_{ts}(x_s), f_{ts}(y_s)) \leq 1+ \eps. \]

Note that for all $f_{ts}(x_s)$ and $f_{ts}(y_s)$, they are at most a distance of $\frac{\eps}{2}$ from $A_t(\gamma)$. Otherwise, for example, $d(f_{ts}(x_s), \gamma f_{ts}(x_s)) \geq \eps$, which contradicts the fact that the maps are distance non-increasing. 

Thus, for each $t \geq s$ there exist $z_t, w_t$ such that $[z_t, w_t] \subset A_t(\gamma) \cap [f_{ts}(x_s)), f_{ts}(y_s))]$, and $d(z_t,f_{ts}(x_s))<\frac{\eps}{2}$ and $d(w_t, f_{ts}(y_s))<\frac{\eps}{2}$.
Hence, $d(z_t, w_t) \geq 1-\eps$. 

Let $N$ be the number of $f_{\infty s}$-illegal turns in the path $[x_s, y_s]$. Thus, the number of $f_{\infty t}$-illegal turns in the path
$[f_{ts}(x_s)), f_{ts}(y_s))]$ is $\leq N$. Hence, the number of $f_{\infty t}$-illegal turns in the path 
$[z_t, w_t]$ is $\leq N$. Let us denote the points of $[z_t, w_t]$ where the illegal turns occur by $a^1_t, \dots a^N_t$, and we also denote $a^0_t := z_t$ and $a^{N+1}_t:=w_t$. 

However, for sufficiently large $t$, the translation length of $\gamma$ in $T_t$ is $<\frac{1-\eps}{3(N+1)}$. Note that since $[z_t, w_t]$ is on $A(\gamma)$, we have that 
$[z_t, w_t] \cap \gamma [z_t, w_t]$ is equal to $[z_t, w_t]$ with (possibly) segments of lengths $\leq \frac{1-\eps}{3(N+1)}$ cut from either end.
Since there are $\leq N+1$ segments $(a^i_t, a^{i+1}_t)$ in $[z_t, w_t]$, one of them has length $\geq \frac{1-\eps}{N+1}$. 
Thus, for some $i=0,1, \dots,  N$, we have $\gamma(a^i_t) \in (a^i_t, a^{i+1}_t)$ or $\gamma^{-1}(a^i_t) \in (a^{i-1}_t, a^{i}_t)$, which are both legal segments. Without generality loss suppose the former. Thus, the turn taken by 
$[z_t, w_t]$ at $a^i_t$ is illegal, but the turn taken by $\gamma[z_t,w_t]$ at $\gamma(a^i_t)$ is legal, since it equals the turn taken by $[z_t, w_t]$ at $\gamma(a^i_t)$. This is a contradiction to the equivariance.
\end{proof}

\bibliographystyle{amsalpha}
\bibliography{PaperReferences}

\end{document}